\documentclass[a4paper,10pt]{article}

\usepackage[utf8]{inputenc}
\usepackage[english]{babel}
\usepackage{amsthm,amssymb,amsfonts,amsmath}
\usepackage[authoryear]{natbib}
\usepackage{dsfont}
\usepackage{color,fancybox,graphicx}
\usepackage{url}
\usepackage{psfrag}
\usepackage{color}
\usepackage{pifont}
\usepackage{latexsym,pstricks} 

\usepackage{hyperref}
\usepackage{mathrsfs }
\usepackage{pst-tree}
\usepackage{fancyhdr}
\usepackage{etoolbox}
\usepackage{setspace}
\usepackage{enumitem}
\usepackage[ruled,linesnumbered]{algorithm2e}

\newcommand{\T}{\Theta}

\newcommand{\E}{\mathds{E}}

\renewcommand{\P}{\mathds{P}}

\newcommand{\bX}{{\bf X}}
\newcommand{\bx}{{\bf x}}
\newcommand{\x}{{\bf x}}
\newcommand{\X}{{\bf X}}
\newcommand{\bz}{{\bf z}}
\newcommand{\z}{{\bf z}}
\newcommand{\0}{{\bf 0}}

\newcommand{\diff}{\mbox{\normalfont d}}

\newtheorem{theorem}{Theorem}[section]
\newtheorem{lemme}{Lemma}
\newtheorem{proposition}{Proposition}

\newtheoremstyle{break}  
  {\topsep}   
  {\topsep}   
  {\itshape}  
  {0pt}       
  {\bfseries} 
  {}         
  {5pt plus 1pt minus 1pt}  
  {}          
\theoremstyle{break}
\newtheorem*{assumption}{}

\title{}
\author{}
\date{}
\bibpunct{(}{)}{;}{a}{,}{,}

\usepackage{hyperref}
\hypersetup{
  colorlinks = true,
  urlcolor = blue, 
  linkcolor = blue,
  citecolor = blue,
}

\begin{document}
\begin{center}
{\Large 
\textbf{\textsf{Random forests and kernel methods}}}
\medskip
\medskip
\end{center}

\noindent{\bf Erwan Scornet }\\
{\it Sorbonne Universit\'es, UPMC Univ Paris 06, F-75005, Paris, France}\\
\href{mailto:erwan.scornet@upmc.fr}{erwan.scornet@upmc.fr}\

\medskip
\begin{abstract}
\noindent {\rm 
Random forests are ensemble methods which grow trees as base learners and combine their predictions by averaging. Random forests are known for their good practical performance, particularly in high-dimensional settings. On the theoretical side, several studies highlight the potentially fruitful connection between random forests and kernel methods. In this paper, we work out in full details this connection. In particular, we show that by slightly modifying their definition, random forests can be rewritten as kernel methods (called KeRF for Kernel based on Random Forests) which are more interpretable and easier to analyze. Explicit expressions of KeRF estimates for some specific random forest models are given, together with upper bounds on their rate of consistency. We also show empirically that KeRF estimates compare favourably to random forest estimates. 

\medskip

\noindent \emph{Index Terms} --- Random forests, randomization, consistency, rate of consistency, kernel methods.

\medskip

\noindent \emph{2010 Mathematics Subject Classification}: 62G05, 62G20.}
\end{abstract}

\section{Introduction}

Random forests are a class of learning algorithms used to solve pattern recognition problems. As ensemble methods, they grow many trees as base learners and aggregate them to predict. Growing many different trees from a single data set requires to randomize the tree building process by, for example, sampling the data set. Thus, there exists a variety of random forests, depending on how trees are built and how the randomness is introduced in the tree building process. 

One of the most popular random forests is that of \citet{Br01} which grows trees based on CART procedure \citep[Classification and Regression Trees,][]{BrFrOlSt84} and randomizes both the training set and the splitting directions. Breiman's \citeyearpar{Br01} random forests have been under active investigation during the last decade mainly because of their good practical performance and their ability to handle high dimensional data sets. Moreover, they are easy to run since they only depend on few parameters which are easily tunable \citep[][]{LiWi02, GePoTu08}. 
They are acknowledged to be state-of-the-art methods in fields such as genomics \citep[][]{Qi12} and pattern recognition \citep[][]{RoRiRaOrTo08}, just to name a few.

However, even if random forests are known to perform well in many contexts, little is known about their mathematical properties. Indeed, most authors study forests whose construction does not depend on the data set. Although, consistency of such simplified models has been addressed in the literature \citep[e.g.,][]{BiDeLu08,IsKo10, DeMaFr13}, these results do not adapt to Breiman's forests whose construction strongly depends on the whole training set. The latest attempts to study the original algorithm are by \citet{MeHo14a} and \citet{Wa14} who prove its asymptotic normality or by \citet{ScBiVe14} who prove its consistency under appropriate assumptions. 

Despite these works, several properties of random forests still remain unexplained. A promising way for understanding their complex mechanisms is to study the connection between forests and kernel estimates, that is estimates $m_n$ which take the form 
\begin{align}
m_n(\bx) = \frac{\sum_{i=1}^n Y_i K_k(\bX_i, \bx)}{\sum_{i=1}^n K_k(\bX_i, \bx)}, \label{def_kernel_estimates}
\end{align}
where $\{(\bX_i, Y_i): 1 \leq i \leq n\}$ is the training set, $(K_k)_k$ is a sequence of kernel functions, and $k$ ($k \in \mathds{N}$) is a parameter to be tuned. Unlike the most used Nadaraya-Watson kernels \citep[][]{Na64, Wa64} which satisfy a homogeneous property of the form $K_h(\bX_i, \bx) = K((\bx - \bX_i)/h)$, kernels $K_k$ are not necessarily of this form. Therefore, the analysis of kernel estimates defined by (\ref{def_kernel_estimates}) turns out to be more complicated and cannot be based on general results regarding Nadaraya-Watson kernels. 

\citet{Br00a} was the first to notice the link between forest and kernel methods, a link which was later formalized by \citet{GeErWe06}. On the practical side, \citet{DaGh14} highlight the fact that a specific kernel based on random forests can empirically outperform state-of-the-art kernel methods. Another approach is taken by \citet{LiJe06} who establish the connection between random forests and adaptive nearest neighbor, implying that random forests can be seen as adaptive kernel estimates \citep[see also][]{BiDe10}. The latest study is by \citet{ArGe14} who show that a specific random forest can be written as a kernel estimate and who exhibit rates of consistency. However, despite these works, the literature is relatively sparse regarding the link between forests and kernel methods.

Our objective in the present paper is to prove that a slight modification of random forest procedures have explicit and simple interpretations in terms of kernel methods. Thus, the resulting kernel based on random forest (called  KeRF in the rest of the paper) estimates are more amenable to mathematical analysis. They also appear to be empirically as accurate as random forest estimates. To theoretically support these results, we also make explicit the expression of some KeRF. We prove upper bounds on their rates of consistency, which compare favorably to the existing ones.

The paper is organized as follows. Section $2$ is devoted to notations and to the definition of KeRF estimates. The link between KeRF estimates and random forest estimates is made explicit in Section $3$. In Section $4$, two KeRF estimates are presented and their consistency is proved along with their rate of consistency. Section $5$ contains experiments that highlight the good performance of KeRF  compared to their random forests counterparts. Proofs are postponed to Section $6$.

\section{Notations and first definitions}

\subsection{Notations}

Throughout the paper, we assume to be given a training sample $\mathcal D_n=\{(\bX_1,Y_1),$ $ \hdots, (\bX_n,Y_n)\}$ of $[0,1]^d\times$ $  \mathbb R$-valued independent random variables distributed as the independent prototype pair $(\bX,$ $ Y)$, where $\mathds{E}[Y^2]<\infty$. We aim at predicting the response $Y$, associated with the random variable $\bX$, by estimating the regression function $m(\bx) = \E \left[ Y | \bX = \bx\right]$. In this context, we use infinite random forests (see the definition below) to build an estimate $m_{\infty, n}: [0,1]^d \to \mathds{R}$ of $m$, based on the data set $\mathcal{D}_n$.

A random forest is a collection of $M$ randomized regression trees \citep[for an overview on tree construction, see e.g., Chapter $20$ in][]{GyKoKrWa02}. For the $j$-th tree in the family, the predicted value at point $\bx$ is denoted by $m_n(\bx, \Theta_j)$, where $\Theta_1, \hdots,\Theta_M$ are independent random variables, distributed as a generic random variable $\Theta$, independent of the sample $\mathcal D_n$. 
This random variable can be used to sample the training set or to select the candidate directions or positions for splitting. The trees are combined to form the finite forest estimate 
\begin{align}
m_{M,n}({\bf x}, \Theta_1, \hdots, \Theta_M) = \frac{1}{M} \sum_{j=1}^M m_n(\bx, \Theta_j). \label{finite_forest}
\end{align}
By the law of large numbers, for all $\bx \in [0,1]^d$, almost surely, the finite forest estimate tends to the infinite forest estimate 
\begin{align*}
m_{\infty,n}(\bx) = \E_{\Theta} \left[m_n(\bx, \Theta)\right],
\end{align*}
where $\mathds{E}_{\Theta}$ denotes the expectation with respect to $\Theta$, conditionally on $\mathcal{D}_n$.

As mentioned above, there is a large variety of forests, depending on how trees are grown and how the random variable $\Theta$ influences the tree construction. 
For instance, tree construction can be independent of $\mathcal{D}_n$ \citep[][]{Bi12}. On the other hand, it can depend only on the $\bX_i$'s \citep[][]{BiDeLu08} or on the whole training set \citep[][]{CuZh01,GeErWe06, ZhZeKo12}.
Throughout the paper, we use three important types of random forests to exemplify our results:  Breiman's, centred and uniform forests. In Breiman's original procedure, splits are performed to minimize the variances within the two resulting cells. The algorithm stops when each cell contains less than a small pre-specified number of points \citep[typically between $1$ and $5$; see][for details]{Br01}. Centred forests are a simpler procedure which, at each node, uniformly select a coordinate among $\{1, \hdots, d\}$ and performs splits at the center of the cell along the pre-chosen coordinate. The algorithm stops when a full binary tree of level $k$ is built (that is, each cell is cut exactly $k$ times), where $k\in \mathds{N}$ is a parameter of the algorithm \citep[see][for details on the procedure]{Br04}. Uniform forests are quite similar to centred forests except that once a split direction is chosen, the split is drawn uniformly on the side of the cell, along the preselected coordinate \citep[see, e.g.,][]{ArGe14}. 

%

\subsection{Kernel based on random forests (KeRF)}

To be more specific, random forest estimates satisfy, for all $\bx \in [0,1]^d$,
\begin{align*}
m_{M,n}(\x, \T_1, \hdots, \T_M) =  \frac{1}{M} \sum_{j=1}^M \left( \sum_{i=1}^n \frac{Y_i \mathds{1}_{{\bf X}_i \in A_n(\x, \T_j)}}{N_n(\x, \Theta_j)}\right),
\end{align*}
where $A_n(\bx, \Theta_j)$ is the cell containing $\bx$, designed with randomness $\Theta_j$ and data set $\mathcal{D}_n$, and 
\begin{align*}
N_n(\x, \Theta_j) = \sum_{i=1}^n \mathds{1}_{{\bf X}_i \in A_n(\x, \T_j)}
\end{align*}
is the number of data points falling in $A_n(\bx, \Theta_j)$. Note that, the weights $W_{i,j,n}(\bx)$ of each observation $Y_i$ defined by
\begin{align*}
W_{i,j,n}(\bx) = \frac{\mathds{1}_{{\bf X}_i \in A_n(\x, \T_j)}}{N_n(\x, \Theta_j)}
\end{align*}
depend on the number of observations $N_n(\x, \Theta_j)$. Thus the contributions of observations that are in cells with a high density of data points are smaller than that of observations which belong to less populated cells. This is particularly true for non adaptive forests (i.e.,  forests built independently of data) since the number of observations in each cell cannot be controlled.
Giving important weights to observations that are in low-density cells can potentially lead to rough estimates. Indeed, as an extreme example, trees of non adaptive forests can contain empty cells which leads to a substantial misestimation (since the prediction in empty cells is set, by default, to zero).

In order to improve the random forest methods and compensate the misestimation induced by random forest weights, a natural idea is to consider KeRF estimates defined, for all $\bx \in [0,1]^d$, by 
\begin{align}
\widetilde{m}_{M,n}(\x, \T_1, \hdots, \T_M) =  \frac{1}{\sum_{j=1}^M N_n(\x, \Theta_j)} \sum_{j=1}^M \sum_{i=1}^n Y_i    \mathds{1}_{{\bf X}_i \in A_n(\x, \T_j)}. \label{kerf_estimation}
\end{align}
Note that $\tilde{m}_{M,n}(\x, \T_1, \hdots, \T_M)$ is equal to the mean of the $Y_i$'s falling in the cells containing $\bx$ in the forest. Thus, each observation is weighted by the number of times it appears in the trees of the forests. Consequently, in this setting, an empty cell does not contribute to the prediction.

The proximity between KeRF estimates $\widetilde{m}_{M,n}$ and random forest estimates will be thoroughly discussed in Section $3$. As for now, we focus on (\ref{kerf_estimation}) and start by proving that it is indeed a kernel estimate whose expression is given by Proposition \ref{lemme1}.
\begin{proposition} \label{lemme1}
Almost surely, for all $\bx \in [0,1]^d$, we have
\begin{align}
\widetilde{m}_{M,n}(\x, \T_1, \hdots, \T_M ) = \frac{ \sum_{i=1}^n Y_i K_{M, n}(\x,{\bf X}_i) }{ \sum_{\ell=1}^n  K_{M, n}(\x,{\bf X}_{\ell})}, \label{weightedapproximation}
\end{align}
where 
\begin{align}
K_{M, n}(\x,\z) = \frac{1}{M} \sum_{j=1}^M \mathds{1}_{ \z \in A_n(\x, \Theta_j)}.\label{definition_noyau_fini}
\end{align}
We call $K_{M, n}$ the connection function of the $M$ finite forest.
\end{proposition}
Proposition \ref{lemme1} states that KeRF estimates have a more interpretable form than random forest estimates since their kernels are the connection functions of the forests.  This connection function can be seen as a geometrical characteristic of the cells in the random forest. Indeed, fixing ${\bf X}_i$, the quantity $K_{M, n}(\x,\X_i)$ is nothing but the empirical probability that ${\bf X}_i$ and $\bx$ are connected (i.e. in the same cell) in the $M$ finite random forest. Thus, the connection function is a natural way to build kernel functions from random forests, a fact that had already been noticed by \citet{Br01}. Note that these kernel functions have the nice property of being positive semi-definite, as proved by \citet{DaGh14}.
%

A natural question is to ask what happens to  KeRF estimates when the number of trees $M$ goes to infinity. To this aim, we define infinite KeRF estimates $\widetilde{m}_{\infty,n}$ by, for all  $\bx$, 
\begin{align}
\widetilde{m}_{\infty,n}(\bx) = \lim\limits_{M \to \infty} \widetilde{m}_{M, n}(\bx, \Theta_1, \hdots, \Theta_M). \label{definition_approx_infini}
\end{align}
In addition, we say that an infinite random forest is discrete (resp. continuous) if its connection function $K_{n}$ is piecewise constant (resp. continuous). For example, Breiman forests and centred forests are discrete but uniform forests are continuous. Denote by $\mathds{P}_{\Theta}$ the probability with respect to $\Theta$, conditionally on $\mathcal{D}_n$. Proposition \ref{convergenceversK} extends the results of Proposition \ref{lemme1} to the case of infinite KeRF estimates.
\begin{proposition} \label{convergenceversK} \
Consider an infinite discrete or continuous forest. Then, almost surely, for all $ \x, \z \in [0,1]^d$,
\begin{align*}
\lim\limits_{M \to \infty} K_{M, n}(\x,\z) = K_{n}(\x,\z),
\end{align*}
where 
\begin{align*}
K_{n}(\x,\z) = \P_{\T}\left[ \bz \in A_n(\bx, \Theta) \right].
\end{align*}
We call $K_{n}$ the connection function of the infinite random forest. Thus, for all $\bx \in [0,1]^d$, one has
\begin{align*}
\widetilde{m}_{\infty,n}(\bx) = \frac{\sum_{i=1}^n Y_i K_{n}(\x,\bX_i) }{\sum_{\ell=1}^n K_{n}(\x,\bX_{\ell})}.
\end{align*}
\end{proposition}
This lemma shows that infinite KeRF estimates are kernel estimates with kernel function equal to $K_{n}$. Observing that $K_n(\bx, \bz)$ is the probability that $\bx$ and $\bz$ are connected in the infinite forest, the function $K_n$ characterizes the shape of the cells in the infinite random forest. 

Now that we know the expression of KeRF estimates, we are ready to study how close this approximation is to random forest estimates. This link will be further work out in Section $4$ for centred and uniform KeRF and empirically studied in Section $5$.

\section{Relation between KeRF and random forests}

In this section, we investigate in which cases KeRF and forest estimates are close to each other. To achieve this goal, we will need the following assumption.

\begin{assumption}{\bf (H1)}
Fix $\bx \in [0,1]^d$, and assume that $Y \geq 0$ a.s.. Then, one of the following two conditions holds: 
\begin{itemize}
\item[\hspace{5pt}]{\bf (H1.1)}  There exist sequences $(a_n), (b_n)$ such that, a.s., 
\begin{align*}
& a_n \leq N_n(\x, \T) \leq b_n.
\end{align*}

\item[\hspace{5pt}]{\bf (H1.2)} There exist sequences $(\varepsilon_n), (a_n), (b_n)$ such that, a.s., 
\begin{itemize}[label=$\bullet$]
\item $1 \leq a_n \leq \E_{\T} \left[ N_n(\x, \T) \right] \leq b_n$,
\item $\mathds{P}_{\Theta}\Big[ a_n \leq N_n(\x, \T) \leq b_n \Big] \geq 1 - \varepsilon_n$. 
\end{itemize}

\end{itemize}
\end{assumption}

{\bf (H1)} assumes that the number of points in every cell of the forest can be bounded from above and below. {\bf (H1.1)} holds for finite forests for which the number of points in each cell is controlled almost surely. Typically, {\bf (H1.1)} is verified for adaptive random forests, if the stopping rule is properly chosen. On the other hand, {\bf (H1.2)} holds for infinite forests. Note that the first condition $\E_{\T} \left[ N_n(\x, \T) \right] \geq 1$ in {\bf (H1.2)} is technical and is true if the level of each tree is tuned appropriately.  Several random forests which satisfy {\bf (H1)} are discussed below.

Proposition \ref{lemme2} states that finite forest estimate $m_{M,n }$ and finite KeRF estimate $\widetilde{m}_{M,n}$ are close to each other assuming that {\bf (H1.1)} holds.

\begin{proposition} \label{lemme2}\
Assume that {\bf (H1.1)} is satisfied. Thus, almost surely,
\begin{align*}
\bigg|\frac{m_{M,n }(\x, \T_1, \hdots, \T_M)}{\widetilde{m}_{M,n}(\bx, \T_1, \hdots, \T_M)} - 1 \bigg| \leq   \frac{b_n-a_n}{a_n},
\end{align*}
with the convention that $0/0=1$.
\end{proposition}

Since KeRF estimates are kernel estimates of the form (\ref{def_kernel_estimates}), Proposition \ref{lemme2} stresses that random forests are close to kernel estimates if the number of points in each cell is controlled. As highlighted by the following discussion, the assumptions of Proposition  \ref{lemme2} are satisfied for some types of random forests.

\paragraph{Centred random forests of level $k$.} For this model, whenever $\bX$ is uniformly distributed over $[0,1]^d$, each cell has a Lebesgue-measure of $2^{-k}$. Thus, fixing $\bx \in [0,1]^d$, according to the law of the iterated logarithm, for all $n$ large enough, almost surely, 
\begin{align*}
\left|N_n(\bx, \Theta) - \frac{n}{2^{k}}\right| \leq \frac{\sqrt{2 n \log \log n}}{2}.
\end{align*}
Consequently, {\bf (H1.1)} is satisfied for $a_n = n2^{-k} - \sqrt{2 n \log \log n}/2$ and $b_n = n2^{-k} + \sqrt{2 n \log \log n}/2$. This  yields, according to Proposition \ref{lemme2}, almost surely, 
\begin{align*}
\bigg|\frac{m_{M,n }(\x, \T_1, \hdots, \T_M)}{\widetilde{m}_{M,n}(\bx, \T_1, \hdots, \T_M)} - 1 \bigg| \leq \frac{\sqrt{2 n \log \log n}}{n2^{-k} - \sqrt{2 n \log \log n}/2}.
\end{align*}
Thus, choosing for example $k=(\log_2 n)/3$, centred KeRF estimates are asymptotically equivalent to centred forest estimates as $n \to \infty$. The previous inequality can be extended to the case where $\bX$ has a density $f$ satisfying $c \leq f \leq C$, for some constants $0<c<C<\infty$. In that case, almost surely, 
\begin{align*}
\bigg|\frac{m_{M,n }(\x, \T_1, \hdots, \T_M)}{\widetilde{m}_{M,n}(\bx, \T_1, \hdots, \T_M)} - 1 \bigg| \leq \frac{\sqrt{2 n \log \log n} + (C-c)n/2^k}{nc2^{-k} - \sqrt{2 n \log \log n}/2}.
\end{align*}
However, the right-hand term does not tend to zero as $n\to \infty$, meaning that the uniform assumption on $\bX$ is crucial to prove the asymptotic equivalence of $m_{M,n }$ and  $\widetilde{m}_{M,n}$ in the case of centred forests.

\paragraph{Breiman's forests.} Each leaf in Breiman's trees contains a small number of points (typically between $1$ and $5$). Thus, if each cell contains exactly one point (default settings in classification problems), {\bf (H1.1)} holds with $a_n=b_n= 1$. Thus, according to Proposition \ref{lemme2}, almost surely, 
\begin{align*}
 m_{M,n}(\x, \T_1, \hdots , \T_M)  = \widetilde{m}_{M,n}(\x, \T_1, \hdots , \T_M ). 
\end{align*}
More generally, if the number of observations in each cell varies between $1$ and $5$, one can set $a_n =1$ and $b_n=5$. Thus, still by Proposition \ref{lemme2}, almost surely,
\begin{align*}
\bigg|\frac{m_{M,n }(\x, \T_1, \hdots, \T_M)}{\widetilde{m}_{M,n}(\bx, \T_1, \hdots, \T_M)} - 1 \bigg| \leq 4.
\end{align*}

\paragraph{Median forests of level $k$.} In this model, each cell of each tree is split at the empirical median of the observations belonging to the cell. The process is repeated until every cell is cut exactly $k$ times (where $k\in \mathds{N}$ is a parameter chosen by the user). Thus, each cell contains the same number of points $\pm 2$ \citep[see, e.g.,][for details]{BiDe13}, and, according to Proposition \ref{lemme2}, almost surely, 
\begin{align*}
 \bigg|\frac{m_{M,n }(\x, \T_1, \hdots, \T_M)}{\widetilde{m}_{M,n}(\bx, \T_1, \hdots, \T_M)} - 1 \bigg| \leq \frac{2}{a_n}.
\end{align*}
Consequently, if the level $k$ of each tree is chosen such that $a_n \to \infty$ as $n \to \infty$, median KeRF estimates are equivalent to median forest estimates.

The following lemma extends Proposition \ref{lemme2} to infinite KeRF and forest estimates.

\begin{proposition} \label{lemme5}
Assume that {\bf (H1.2)} is satisfied. Thus, almost surely, 
\begin{align*}
|m_{\infty,n }({\bf x}) - \widetilde{m}_{\infty, n}(\x)| & \leq  \frac{b_n - a_n}{a_n } \widetilde{m}_{\infty, n}(\x) + n \varepsilon_n \Big(\max_{1 \leq i \leq n} Y_i \Big).
\end{align*}
\end{proposition}

Considering inequalities provided in Proposition \ref{lemme5}, we see that infinite KeRF estimates are close to infinite random forest estimates if the number of observations in each cell is bounded (via $a_n$ and $b_n$).

It is worth noticing that controlling the number of observations in each cell while obtaining a simple partition shape is difficult to achieve.  
On the one hand, if the tree construction depends on the training set, the algorithm can be stopped when each leaf contains exactly one point and thus KeRF estimate is equal to random forest estimate. However, in that case, the probability $K_{n}(\x,\z)$ is very difficult to express since the geometry of each tree partitioning strongly depends on the training set.
On the other hand, if the tree construction is independent of the training set, the probability $K_{n}(\x,\z)$ can be made explicit in some cases, for example for centred forests (see Section $5$). However, the number of points in each cell is difficult to control (every leaf cannot contain exactly one point with a non-adaptive cutting strategy) and thus KeRF estimate can be far away from random forest estimate.
Consequently, one cannot deduce an explicit expression for random forest estimates from the explicit expression of KeRF estimates.

\section{Two particular KeRF estimates}
\label{section_approximation_two_nonadaptiverandomforest}

According to Proposition \ref{convergenceversK}, infinite KeRF estimate $\widetilde{m}_{\infty,n}$ depends only on the connection function $K_n$ via the following equation
\begin{align}
\widetilde{m}_{\infty,n}(\bx) = \frac{\sum_{i=1}^n Y_i K_{n}(\x,\bX_i) }{\sum_{\ell=1}^n K_{n}(\x,\bX_{\ell})}.\label{infinite_rf_modified_estimate}
\end{align}
To take one step further into the understanding of KeRF, we study in this section the connection function of two specific infinite random forests. 
We focus on infinite KeRF estimates for two reasons. Firstly, the expressions of infinite KeRF estimates are more amenable to mathematical analysis since they do not depend on the particular trees used to build the forest. Secondly, the prediction accuracy of infinite random forests is known to be better than that of finite random forests \Citep[see, e.g.,][]{Sc14}. Therefore infinite KeRF estimates are likely to be more accurate than finite KeRF estimates.

Practically, both infinite KeRF estimates and infinite random forest estimates can only be approximated by Monte Carlo simulations. Here, we show that centred KeRF estimates have an explicit expression, that is their connection function can be made explicit. Thus, infinite centred KeRF estimates and infinite uniform KeRF estimates (up to an approximation detailed below) can be directly computed using equation (\ref{infinite_rf_modified_estimate}). 
%

\paragraph{Centred KeRF}

As seen above, the construction of centred KeRF of level $k$ is the same as for centred forests of level $k$ except that predictions are made according to equation (\ref{kerf_estimation}). 
Centred random forests are closely related to Breiman's forests in a linear regression framework. Indeed, in this context, splits that are performed at a low level of the trees are roughly located at the middle of each cell. In that case, Breiman's forests and centred forests are close to each other, which justifies the interest for these simplified models, and thus for centred KeRF.

In the sequel, the connection function of the centred random forest of level $k$ is denoted by $K_{k}^{cc}$. This notation is justified by the fact that the construction of centred KeRF estimates depends only on the size of the training set through the choice of $k$.

\begin{proposition}
\label{lemme_centred_random_forest}
Let $k \in \mathds{N}$ and consider an infinite centred random forest of level $k$. Then, for all $\x, \z \in [0,1]^d$,
\begin{align*}
K_{k}^{cc}(\x,\z) = \sum\limits_{\substack{k_{1},\hdots,k_{d} \\ \sum_{\ell=1}^d k_{\ell} = k }} 
\frac{k!}{k_{1}! \hdots k_{d} !} \left( \frac{1}{d}\right)^k \prod_{j=1}^d  \mathds{1}_{ \lceil 2^{k_j}x_j \rceil = \lceil 2^{k_j}z_j \rceil}.
\end{align*}
\end{proposition}

Note that ties are broken by imposing that cells are of the form $\prod_{i=1}^d A_i$ where the $A_i$ are equal to $]a_i, b_i]$ or $[0,b_i]$, for all $0<a_i < b_i \leq 1$. Figure \ref{FigureNoyauCentre} shows a graphical representation of the function $f$ defined as 
\begin{align*}
\begin{array}{cccl}
f_k: 		& 	 [0,1]  \times  [0,1]     & \to    & [0,1]\\
			 & 	 \bz = (   z_1           ,     z_2    ) & \mapsto & K_k^{cc}\big((\frac{1}{2},\frac{1}{2}),\bz\big).
\end{array}
\end{align*}

\begin{figure}[h!]
 \hspace{-0.4cm}
\begin{tabular}{ccc}
\includegraphics[scale=0.19]{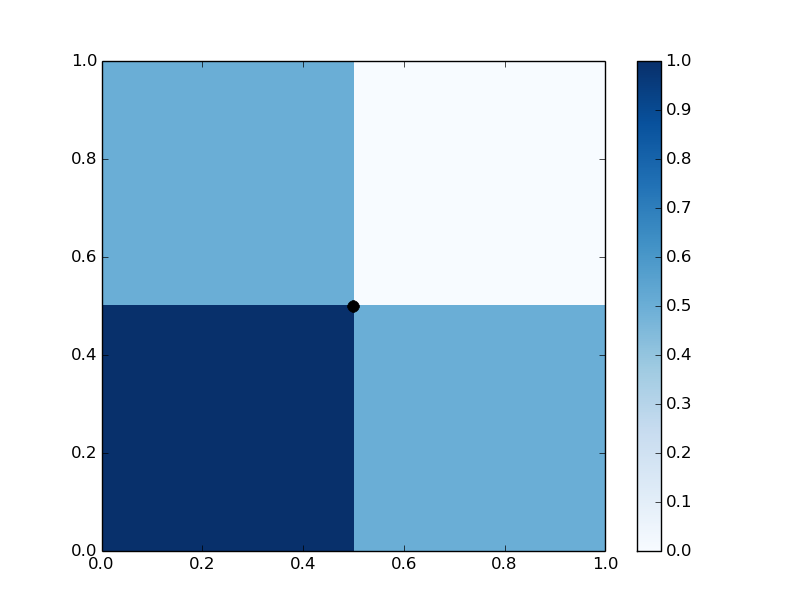}  &
\includegraphics[scale=0.19]{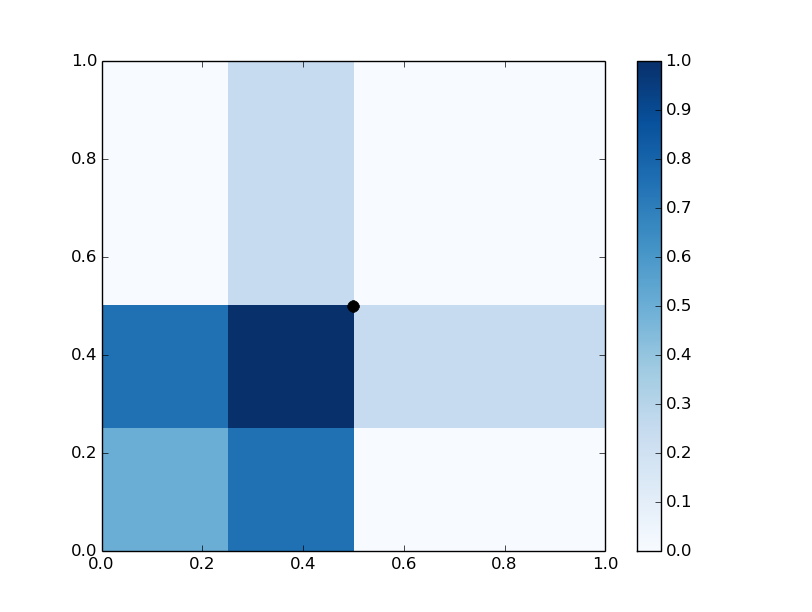}  &
 \includegraphics[scale=0.19]{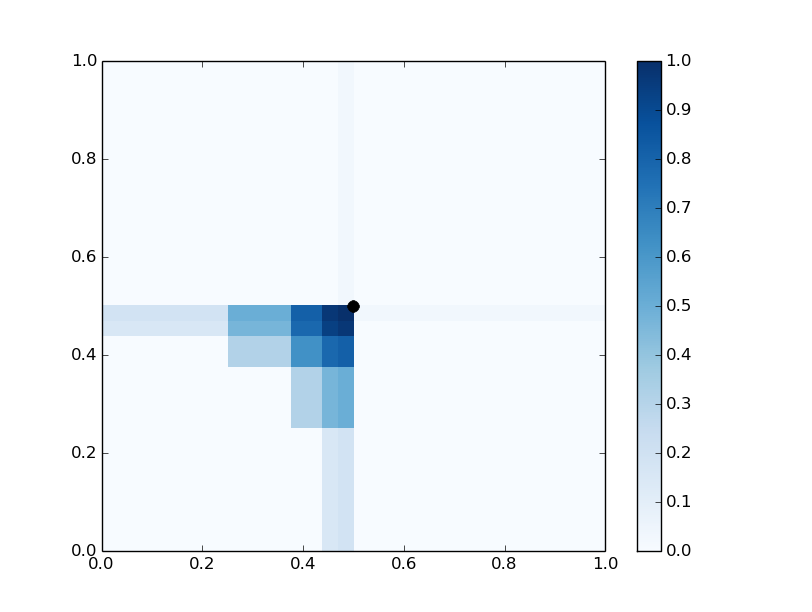} 
\end{tabular}
\caption{Representations of $f_1$, $f_2$ and $f_5$ in $[0,1]^2$ \label{FigureNoyauCentre}}
\end{figure}

Denote by $\widetilde{m}_{\infty,n}^{cc}$ the infinite centred KeRF estimate, associated with the connection function $K_k^{cc}$, defined as 
\begin{align*}
\widetilde{m}_{\infty,n}^{cc}(\bx) = \frac{\sum_{i=1}^n Y_i K_k^{cc}(\x,\bX_i) }{\sum_{\ell=1}^n K_k^{cc}(\x,\bX_{\ell})}.
\end{align*}
To pursue the analysis of $\widetilde{m}_{\infty,n}^{cc}$, we will need the following assumption on the regression model. 
\begin{assumption}\label{assumption_last_th} {\bf (H2)}
One has 
\begin{align*}
Y = m(\bX) + \varepsilon,
\end{align*}
where $\varepsilon$ is a centred Gaussian noise, independent of $\bX$, with finite variance $\sigma^2<\infty$. Moreover, $\bX$ is uniformly distributed on $[0,1]^d$ and $m$ is Lipschitz.
\end{assumption}
Our theorem states that infinite centred KeRF estimates are consistent whenever {\bf (H2)} holds. Moreover, it provides an upper bound on the rate of consistency of centred KeRF. 

\begin{theorem}
\label{theoreme_consistency_centred_forest_approximation}
Assume that {\bf (H2)} is satisfied. Then, providing $k \to \infty$ and $n/2^k \to \infty$, there exists a constant $C_1>0$ such that, for all $n>1$, and for all $\bx \in [0,1]^d$, 
\begin{align*}
\mathds{E} \left[ \widetilde{m}^{cc}_{\infty, n}(\bx) - m(\bx)\right]^2 \leq C_1 n^{-1/(3 + d\log 2)} (\log n)^2.
\end{align*}
\end{theorem}

Observe that centred KeRF estimates fail to reach minimax rate  of consistency $n^{-2/(d+2)}$ over the class of Lipschitz functions. A similar upper bound on the rate of consistency $n^{-3/4d \log 2 +3}$ of centred random forests was obtained by \citet{Bi12}. 
It is worth noticing that, for all $d\geq 9$,  the upper bound on the rate of centred KeRF is sharper than that of centred random forests. This theoretical result supports the fact that KeRF procedure has a better performance compared to centred random forests. This will be supported by simulations in Section $5$ (see Figure \ref{figure_2})

\paragraph{Uniform KeRF}

Recall that the infinite uniform KeRF estimates of level $k$ are the same as infinite uniform forest of level $k$ except that predictions are  computed according to equation (\ref{kerf_estimation}).
Uniform random forests, first studied by \citet{BiDeLu08}, remain under active investigation. They are a nice modelling of Breiman forests, since with no a priori on the split location, we can consider that splits are drawn uniformly on the cell edges. Other related versions of these forests have been thoroughly investigated by \citet{ArGe14} who compare the bias of a single tree to that of the whole forest.

As for the connection function of centred random forests, we use the notational convention $K_{k}^{uf}$ to denote the connection function of uniform random forests of level $k$. 

\begin{proposition}\label{lemme_foret_uniforme_expressiondunoyau}

Let $k \in \mathds{N}$ and consider an infinite uniform random forest of level $k$. Then, for all $\x \in [0,1]^d$,
\begin{align*}
K_{k}^{uf}({\bf 0},\bx) = \sum\limits_{\substack{k_{1},\hdots,k_{d} \\ \sum_{\ell=1}^d k_{\ell} = k }} \frac{k!}{k_{1}! \hdots k_{d} !}  \left( \frac{1}{d}\right)^k \prod_{m=1}^d  \left( 1 - x_m \sum_{j=0}^{k_{m}-1} \frac{(- \ln x_m)^j}{j!}
 \right),
\end{align*}
with the convention $\sum_{j=0}^{-1} \frac{(- \ln x_m)^j}{j!} = 0$.
\end{proposition}
Proposition \ref{lemme_foret_uniforme_expressiondunoyau} gives the explicit expression of $K_{k}^{uf}({\bf 0},\bx)$. Figure \ref{FigureNoyauUniforme} shows a representation of the functions $f_1$, $f_2$ and $f_5$ defined as
\begin{align*}
\begin{array}{cccl}
f_k: 		& 	 [0,1]  \times  [0,1]     & \to    & [0,1]\\
			& 	 \bz = (   z_1           ,     z_2    ) & \mapsto & K_k^{uf}\big({\bf 0}, \big|\bz - (\frac{1}{2},\frac{1}{2}) \big| \big),
\end{array}
\end{align*}
where $|\bz-\bx| = (|z_1-x_1|, \hdots, |z_d-x_d|)$.
\begin{figure}[h!]
 \hspace{-0.4cm}
\begin{tabular}{ccc}
\includegraphics[scale=0.19]{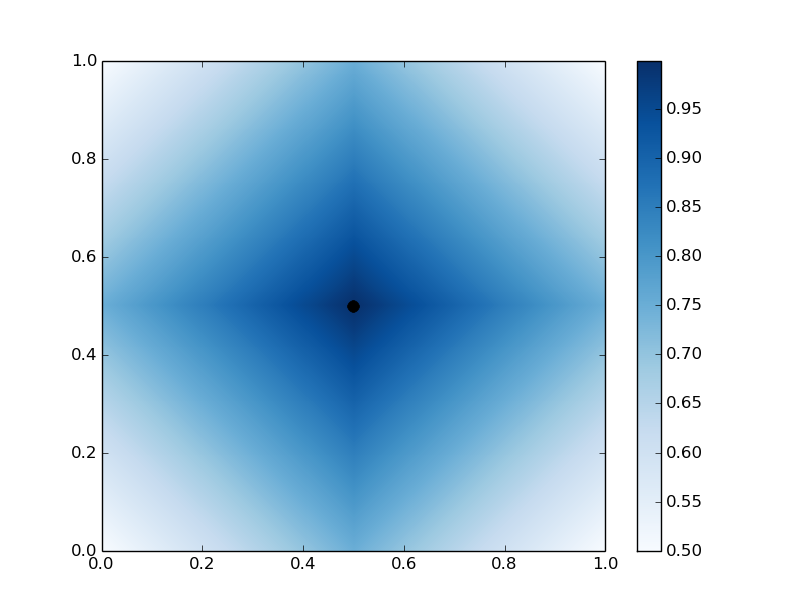}  &
\includegraphics[scale=0.19]{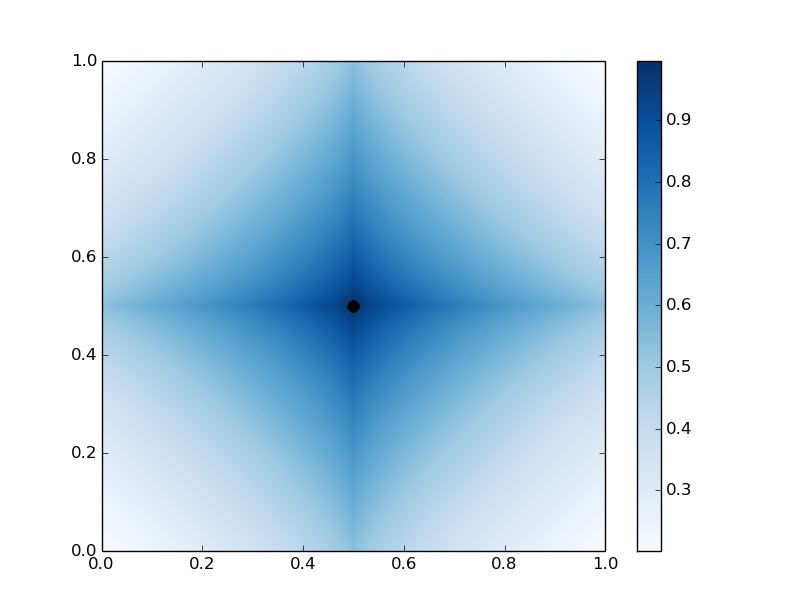}  &
 \includegraphics[scale=0.19]{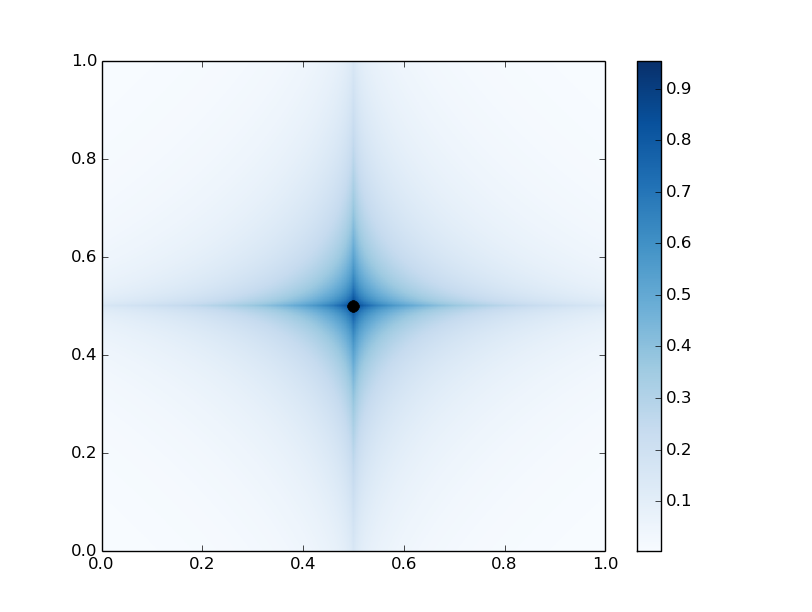} 
\end{tabular}
\caption{Representations of $f_1$, $f_2$ and $f_5$ in dimension two \label{FigureNoyauUniforme}}
\end{figure}

Unfortunately, the general expression of the connection function  $K_{k}^{uf}(\bx,\bz)$ is difficult to obtain. Indeed, for $d=1$, cuts are performed along a single axis, but the probability of connection between two points $x$ and $z$ does not depend only upon the distance $|z-x|$ but rather on the positions $x$ and $z$, as stressed in the following Lemma.

\begin{lemme}\label{lemme_inutile}
Let $x,z \in [0,1]$. Then, 
\begin{align*}
K_1^{uf}(x,z) & = 1 - |z-x|,\\
K_2^{uf}(x,z) & = 1 - |z-x| + |z-x| \log \left( \frac{z}{1-x}\right).
\end{align*}
\end{lemme}

A natural way to deal with this difficulty is to replace the connection function $K^{uf}_k$ by the function $(\bx, \bz) \to  K^{uf}_k({\bf 0}, |\bz - \bx|)$. Indeed, this is a simple manner to build an invariant-by-translation version of the uniform kernel $K^{uf}_k$. The extensive simulations in Section $5$ support the fact that estimates of the form (\ref{infinite_rf_modified_estimate}) built with these two kernels have similar prediction accuracy. As for infinite centred KeRF estimates, we denote by $\widetilde{m}^{uf}_{\infty,n}$ the infinite uniform KeRF estimates but built with the invariant-by-translation version of $K^{uf}_k$, namely
\begin{align*}
\widetilde{m}_{\infty,n}^{uf}(\bx) = \frac{\sum_{i=1}^n Y_i K_k^{uf}(\0,|\bX_i- \bx|) }{\sum_{\ell=1}^n K_k^{uf}(\0,|\bX_{\ell}- \bx|)}.
\end{align*}
Our last theorem states the consistency of infinite uniform KeRF estimates along with an upper bound on their rate of consistency.

\begin{theorem}
\label{theoreme_consistency_uniform_forest_approximation}
Assume that {\bf (H2)} is satisfied. Then, providing $k \to \infty$ and $n/2^k \to \infty$, there exists a constant $C_1>0$ such that, for all $n>1$ and for all $\bx \in [0,1]^d$, 
\begin{align*}
\mathds{E} \left[ \widetilde{m}^{uf}_{\infty,n}(\bx) - m(\bx)\right]^2 \leq C_1 n^{-2/(6 + 3d\log 2)} (\log n)^2.
\end{align*}
\end{theorem}

As for centred KeRF estimates, the rate of consistency does not reach the minimax rate on the class of Lipschitz functions, and is actually worse than that of centred KeRF estimates, whatever the dimension $d$ is. Besides, centred KeRF estimates have better performance than uniform KeRF estimates and this will be highlighted by simulations (Section $5$).

Although centred and uniform KeRF estimates are kernel estimates of the form (\ref{def_kernel_estimates}), the usual tools used to prove consistency and to find rate of consistency of kernel methods cannot be applied here \citep[see, e.g., Chapter $5$ in][]{GyKoKrWa02}. Indeed, the support of $\bz \mapsto K_{k}^{cc}(\bx, \bz)$ and that of $\bz \mapsto K_{k}^{uf}(\0, |\bz-\bx|)$ cannot be contained in a ball centred on $\bx$, whose diameter tends to zero (see Figure \ref{FigureNoyauCentre} and \ref{FigureNoyauUniforme}). The proof of Theorem  \ref{theoreme_consistency_centred_forest_approximation} and \ref{theoreme_consistency_uniform_forest_approximation} are  then based on the previous work of  \citet{GrKrPa84} who proved the consistency of kernels with unbounded support. In particular, we use their bias/variance decomposition of kernel estimates to exhibit upper bounds on the rate of consistency.

\section{Experiments}

Practically speaking, Breiman's random forests are among the most widely used forest algorithms. Thus a natural question is to know whether Breiman KeRF compare favourably to Breiman's forests. In fact, as seen above, the two algorithms coincide whenever Breiman's forests are fully grown. But this is not always the case since by default, each cell of Breiman's forests contain between $1$ and $5$ observations.

We start this section by comparing Breiman KeRF and Breiman's forest estimates for various  regression models described below.  
Some of these models are toy models ({\bf Model 1, 5-8}). {\bf Model 2} can be found in \citet{VaPoHu07} and  {\bf Models 3-4} are presented in \citet{MeVaBu09}. For all regression frameworks, we consider covariates $\bX = (X_1, \hdots, X_d)$ that are uniformly distributed over $[0,1]^d$. We also let $\widetilde{X}_i = 2(X_i-0.5)$ for $1 \leq i \leq d$. 

\begin{itemize}[label = ,leftmargin=0cm]

\item{\bf Model 1}: $n = 800, d = 50, Y = \widetilde{X}_1^2 + \exp(-\widetilde{X}_2^2)$

\item {\bf Model 2}: $n = 600, d = 100, Y = \widetilde{X}_1 \widetilde{X}_2 + \widetilde{X}_3^2 - \widetilde{X}_4 \widetilde{X}_7 + \widetilde{X}_8 \widetilde{X}_{10} - \widetilde{X}_6^2 + \mathcal{N}(0,0.5)$

\item {\bf Model 3}: $n = 600, d = 100, Y = -\sin(2 \widetilde{X}_1) + \widetilde{X}_2^2 + \widetilde{X}_3 - \exp(-\widetilde{X}_4) + \mathcal{N}(0,0.5)$

\item {\bf Model 4}: $n = 600, d = 100, Y = \widetilde{X}_1 + (2 \widetilde{X}_2-1)^2 + \sin(2 \pi \widetilde{X}_3) / (2-\sin(2 \pi \widetilde{X}_3)) + \sin(2 \pi \widetilde{X}_4) + 2 \cos(2\pi \widetilde{X}_4) + 3 \sin^2(2\pi \widetilde{X}_4 )+ 4 \cos^2(2\pi \widetilde{X}_4) + \mathcal{N}(0,0.5)$

\item {\bf Model 5}: $n = 700, d = 20, Y = \mathds{1}_{\widetilde{X}_1 > 0 } + \widetilde{X}_2^3 +  \mathds{1}_{\widetilde{X}_4 + \widetilde{X}_6 - \widetilde{X}_8  - \widetilde{X}_9 > 1 + \widetilde{X}_{10} } + \exp(-\widetilde{X}_2^2) + \mathcal{N}(0,0.5)$

\item {\bf Model 6}: $n = 500, d = 30, Y = \sum_{k=1}^{10} \mathds{1}_{\widetilde{X}_k^3 < 0 } - \mathds{1}_{\mathcal{N}(0,1)> 1.25 }$

\item {\bf Model 7}: $n = 600, d = 300, Y = \widetilde{X}_1^2 + \widetilde{X}_2^2 \widetilde{X}_3  \exp(-|\widetilde{X}_4|) + \widetilde{X}_6 - \widetilde{X}_8 + \mathcal{N}(0,0.5)$


\item {\bf Model 8}: $n= 500,d= 1000,Y=\widetilde{X}_1+3\widetilde{X}_3^2 - 2 \exp(-\widetilde{X}_5) + \widetilde{X}_6$


\end{itemize}

All numerical implementations have been performed using the free Python software, available online at \url{https://www.python.org/}. For each experiment, the data set is divided into a training set ($80\%$ of the data set) and a test set (the remaining $20\%$). Then, the empirical risk ($\mathds{L}^2$ error) is evaluated on the test set.

\begin{sloppypar} 
To start with, Figure \ref{figure_1} depicts the empirical risk of Breiman's forests and Breiman KeRF estimates for two regression models (the conclusions are similar for the remaining regression models). Default settings were used for Breiman's forests ($\texttt{minsamplessplit}=2$ , $\texttt{maxfeatures}=0.333$) and for Breiman KeRF, except that we did not bootstrap the data set. 
Figure \ref{figure_1} puts in evidence that Breiman KeRF estimates behave similarly (in terms of empirical risk) to Breiman forest estimates.  It is also interesting to note that bootstrapping the data set does not change the performance of the two algorithms.
\end{sloppypar}

\begin{figure}[h!]
\begin{center}
\begin{tabular}{ccc}
\includegraphics[width=0.45\textwidth]{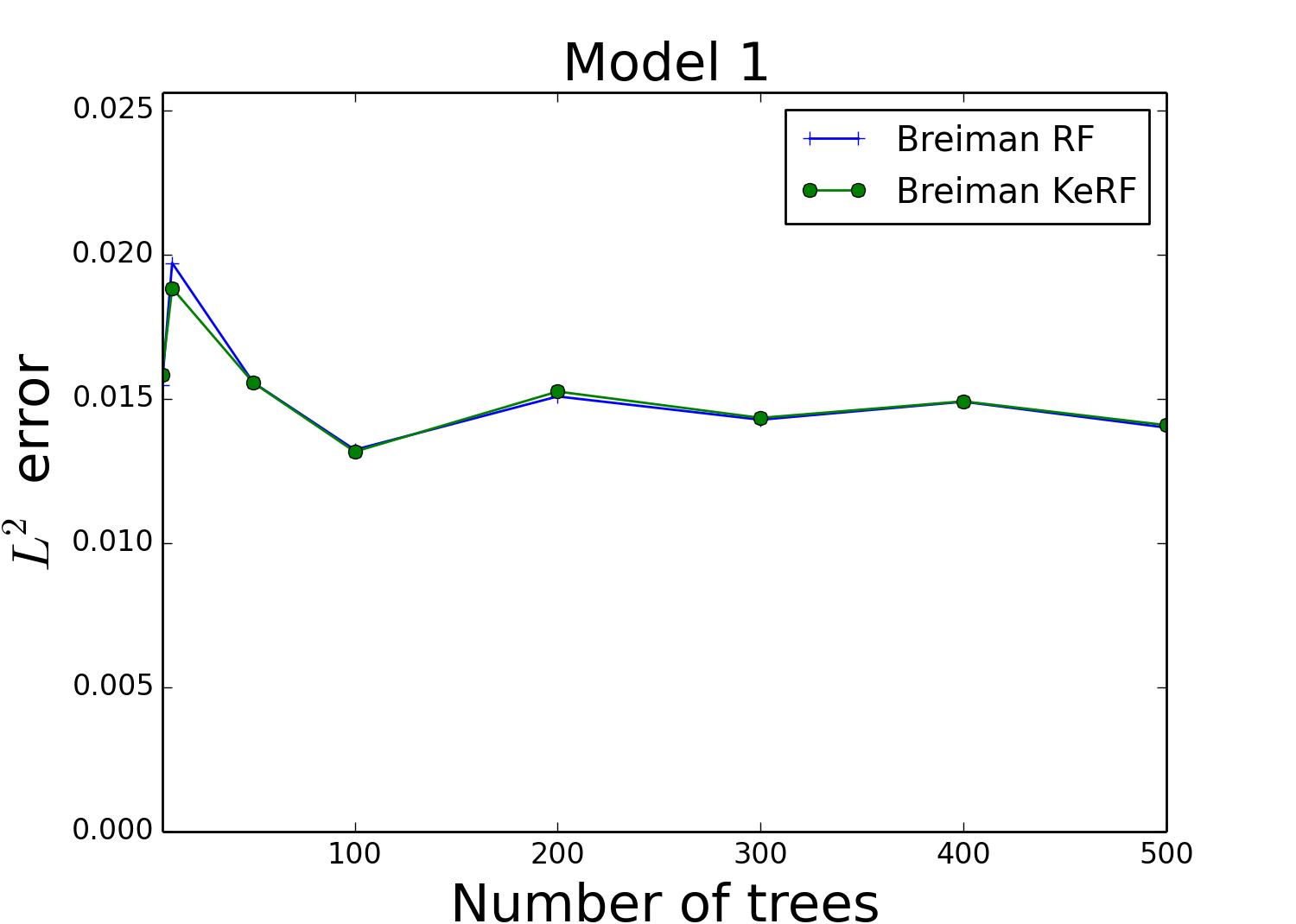}
& \includegraphics[width=0.45\textwidth]{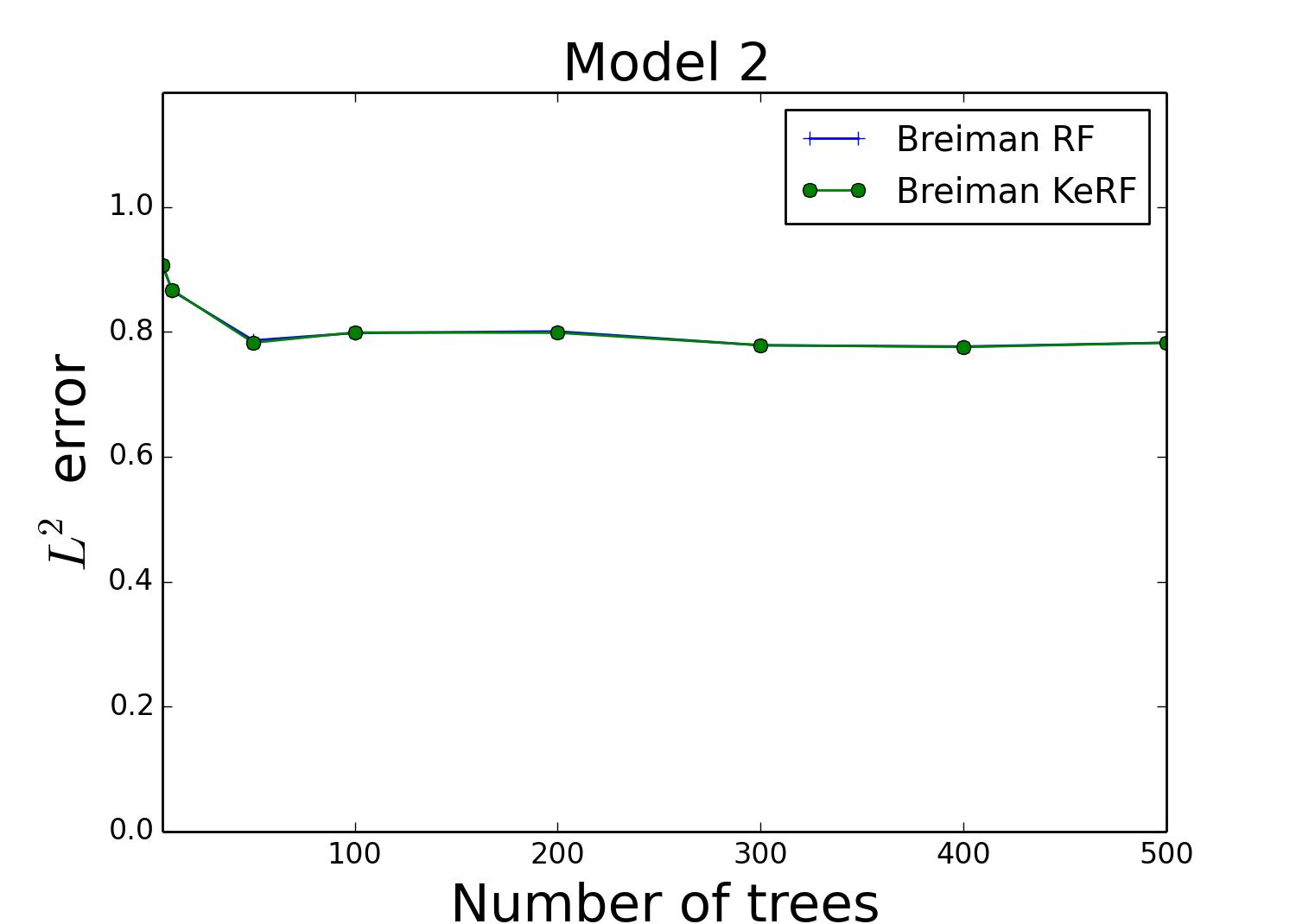}
\end{tabular}
\caption{Empirical risks of Breiman KeRF estimates and Breiman forest estimates.}
\label{figure_1}
 \end{center}
\end{figure}

Figure \ref{figure_3} (resp. Figure \ref{figure_2}) shows the risk of uniform (resp. centred) KeRF estimates compared to the risk of uniform (resp. centred) forest estimates (only two models shown). In these two experiments, uniform and centred forests and their KeRF counterparts have been grown in such a way that each tree is a complete binary tree of level $k = \lfloor \log_2 n \rfloor$. Thus, in that case, each cell contains on average $n/2^{k} \simeq 1$ observation. Once again, the main message of Figure \ref{figure_3}  is that the uniform KeRF accuracy is close to the uniform forest accuracy.

\begin{figure}[h!]
\begin{center}
\begin{tabular}{ccc}
\includegraphics[width=0.45\textwidth]{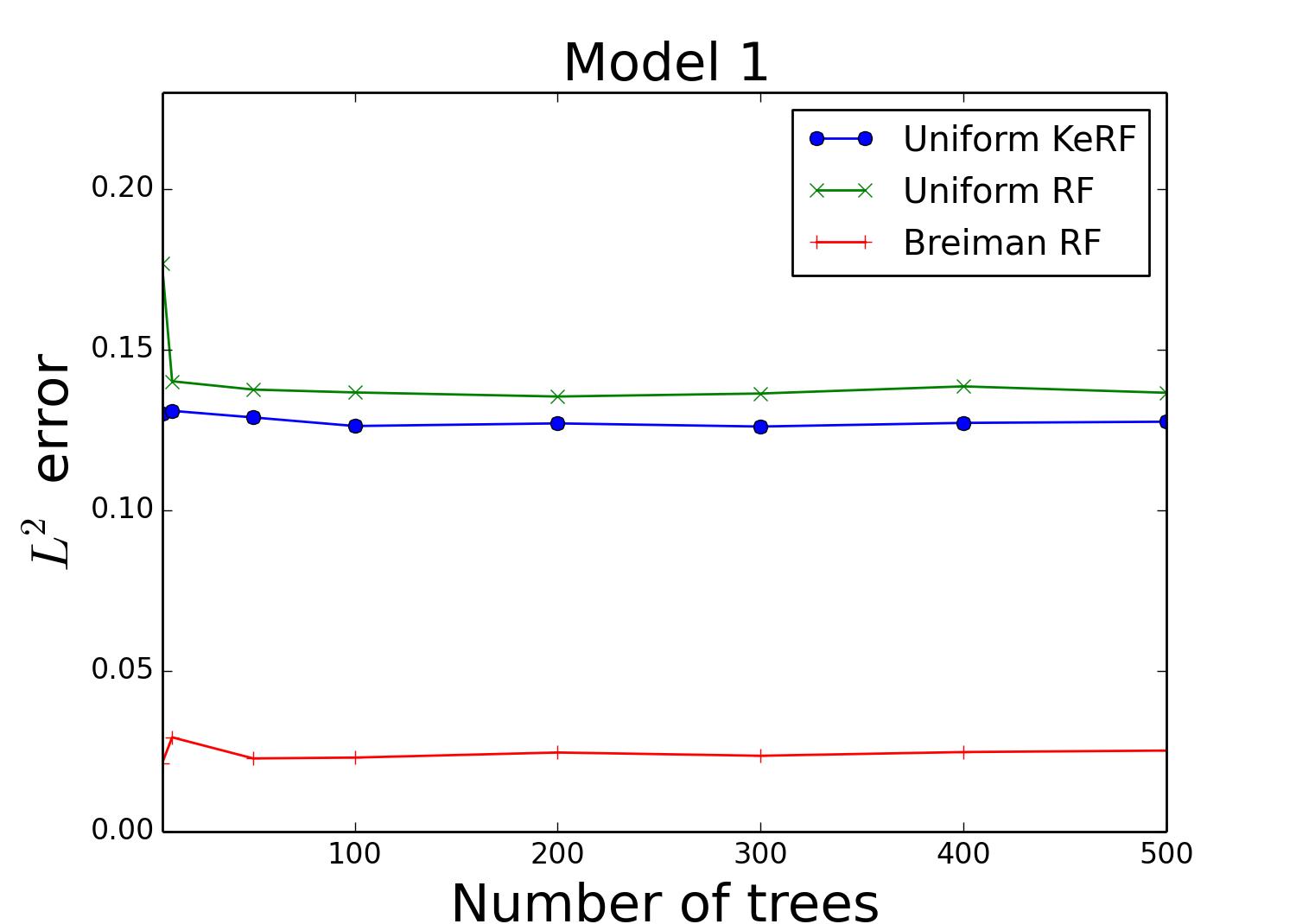} 
& \includegraphics[width=0.45\textwidth]{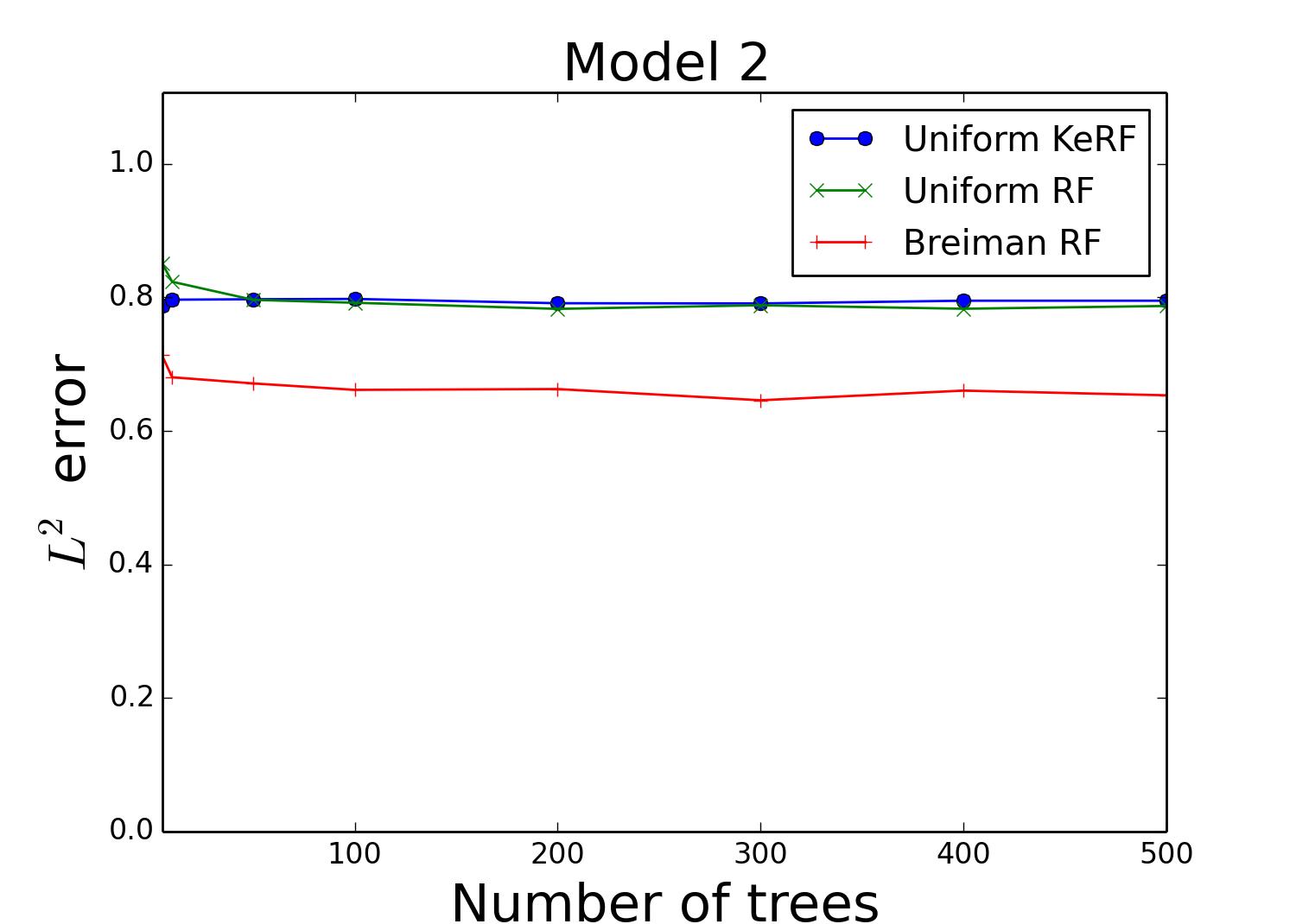}
\end{tabular}
\caption{Empirical risks of uniform KeRF and uniform forest.}
\label{figure_3}
\end{center}
\end{figure}

On the other hand, it turns out that the performance of centred KeRF and centred forests are not similar (Figure \ref{figure_2}). In fact, centred KeRF estimates are either comparable to centred forest estimates (as, for example, in {\bf Model 2}), or have a better accuracy (as, for example, in {\bf Model 1}). 
\begin{figure}[h!]
\begin{center}
\begin{tabular}{cc}
\includegraphics[width=0.45\textwidth]{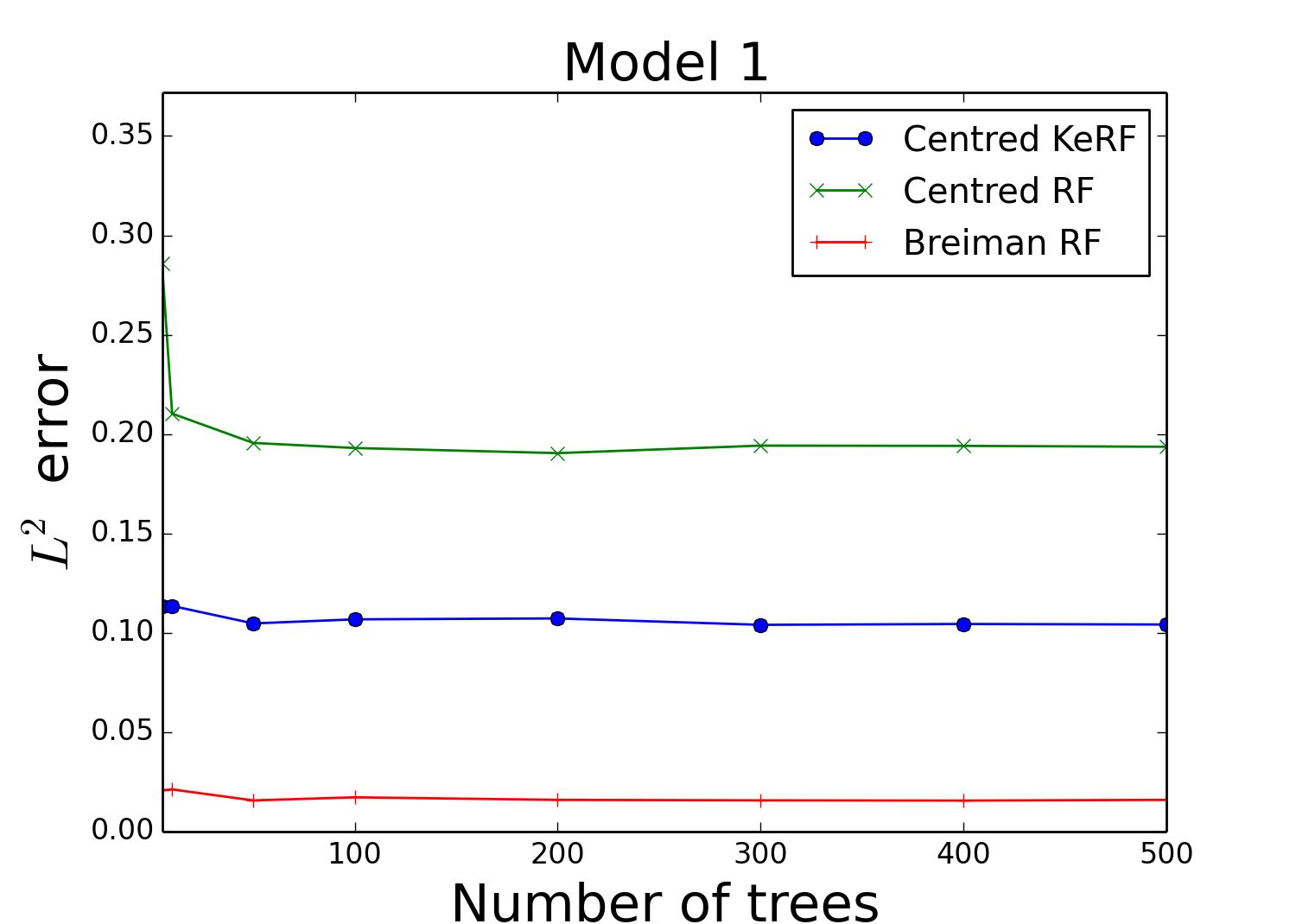}
& \includegraphics[width=0.45\textwidth]{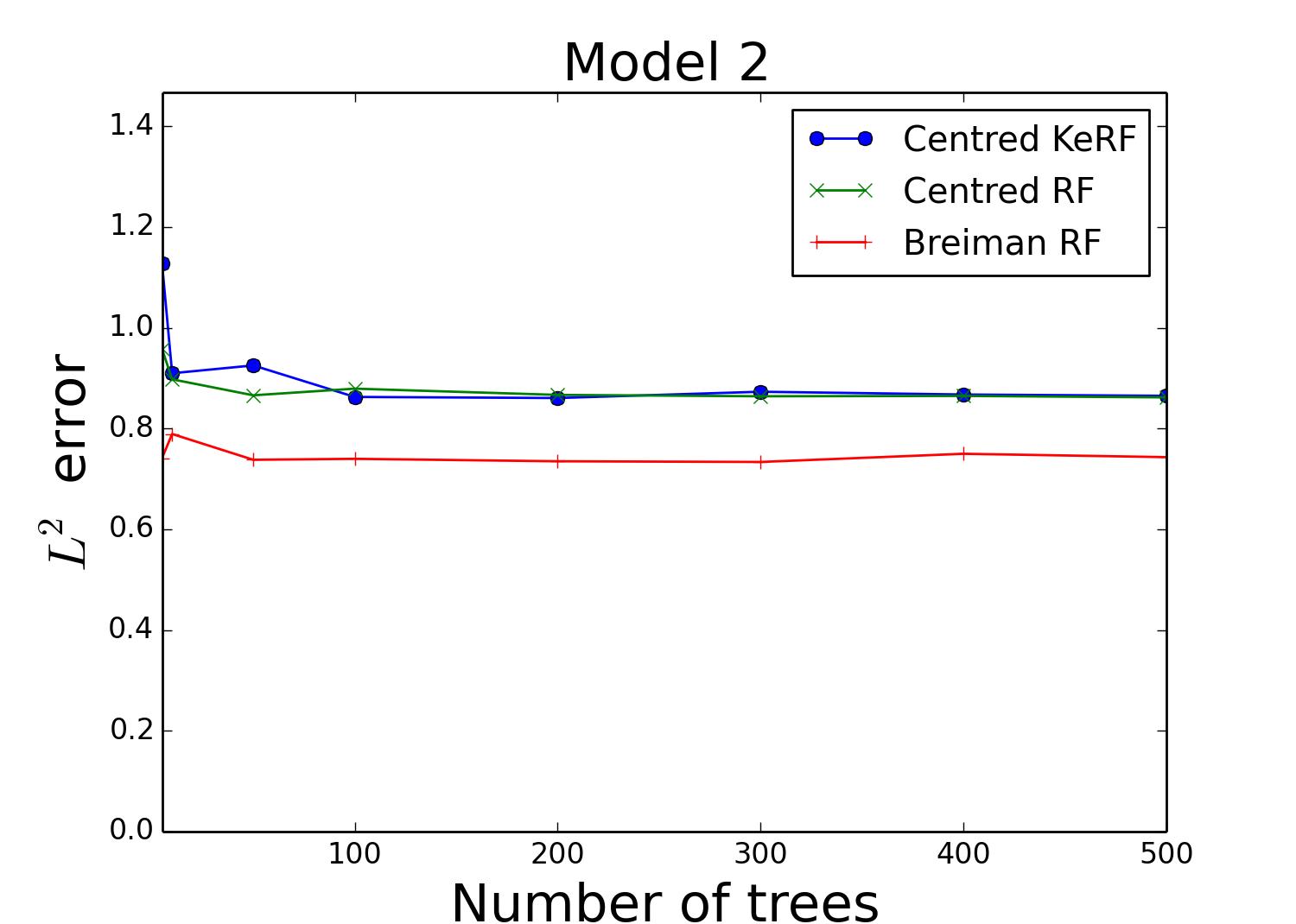} \\
\end{tabular}
\caption{Empirical risks of centred KeRF and centred forest.}
\label{figure_2}
\end{center}
\end{figure}
A possible explanation for this phenomenon is that centred forests are non-adaptive in the sense that their construction does not depend on the data set. Therefore, each tree is likely to contain cells with unbalanced number of data points, which can result in random forest misestimation. This undesirable effect vanishes using  KeRF methods since they assign the same weights to each observation.

The same series of experiments were conducted, but using bootstrap for computing both KeRF and random forest estimates. The general finding is that the results are similar---Figure \ref{figure_6} and \ref{figure_7}  depict the accuracy of corresponding algorithms for a selected choice of regression frameworks. 

\begin{figure}[h!]
\begin{center}
\begin{tabular}{ccc}
\includegraphics[width=0.45\textwidth]{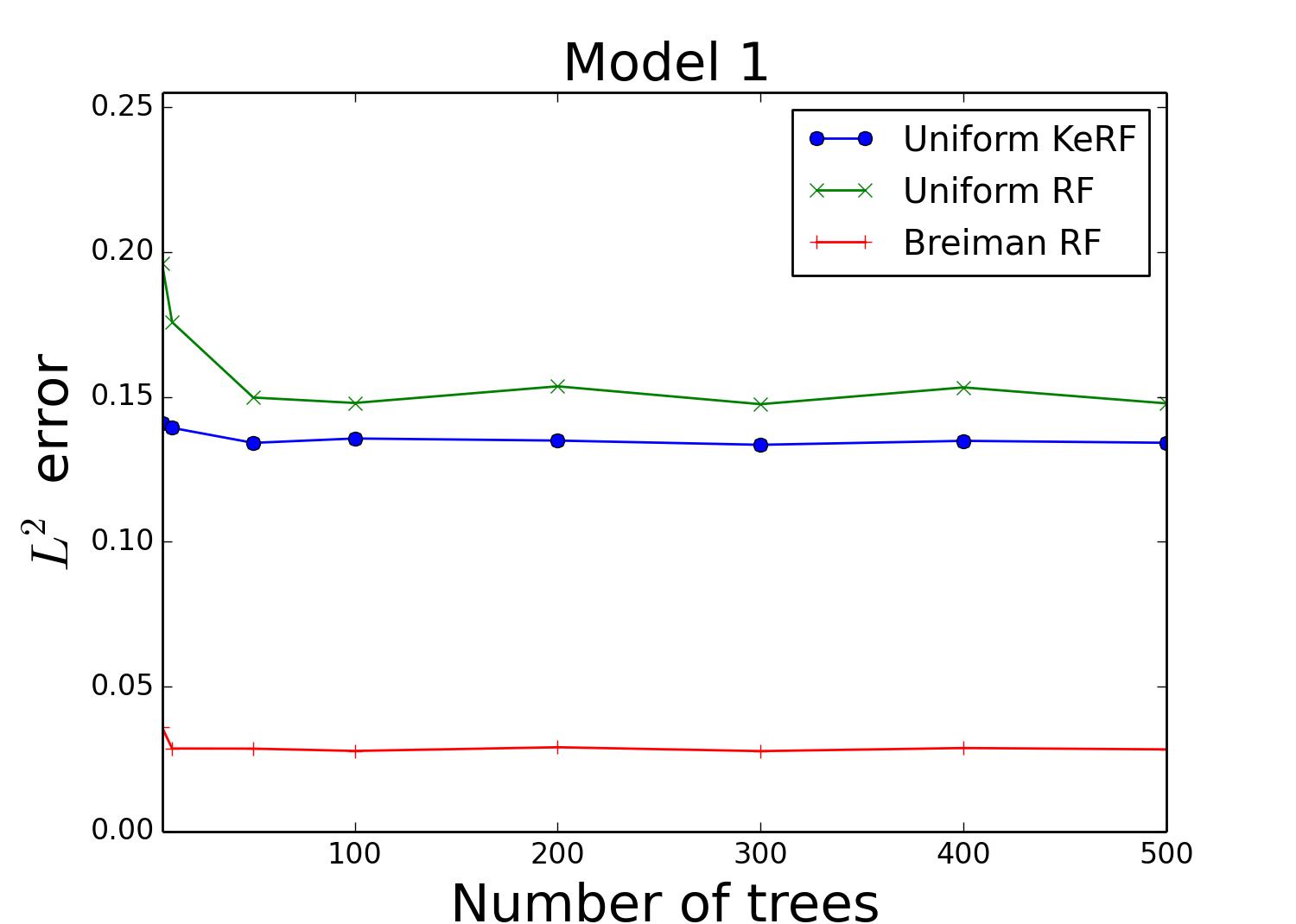} 
& \includegraphics[width=0.45\textwidth]{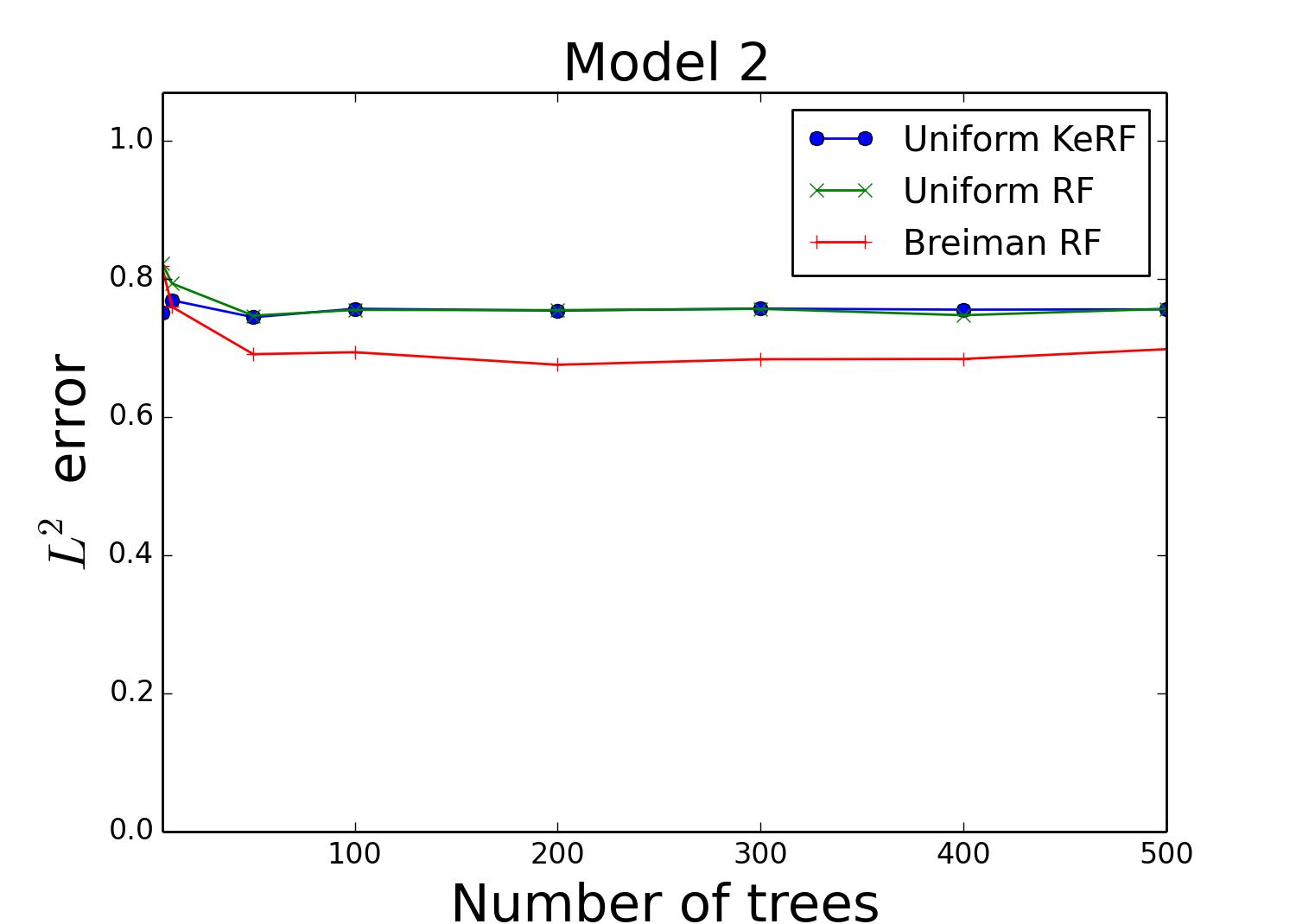}
\end{tabular}
\caption{Empirical risks of uniform KeRF and uniform forest (with bootstrap).}
\label{figure_6}
\end{center}
\end{figure}

\begin{figure}[h!]
\begin{center}
\begin{tabular}{ccc}
\includegraphics[width=0.45\textwidth]{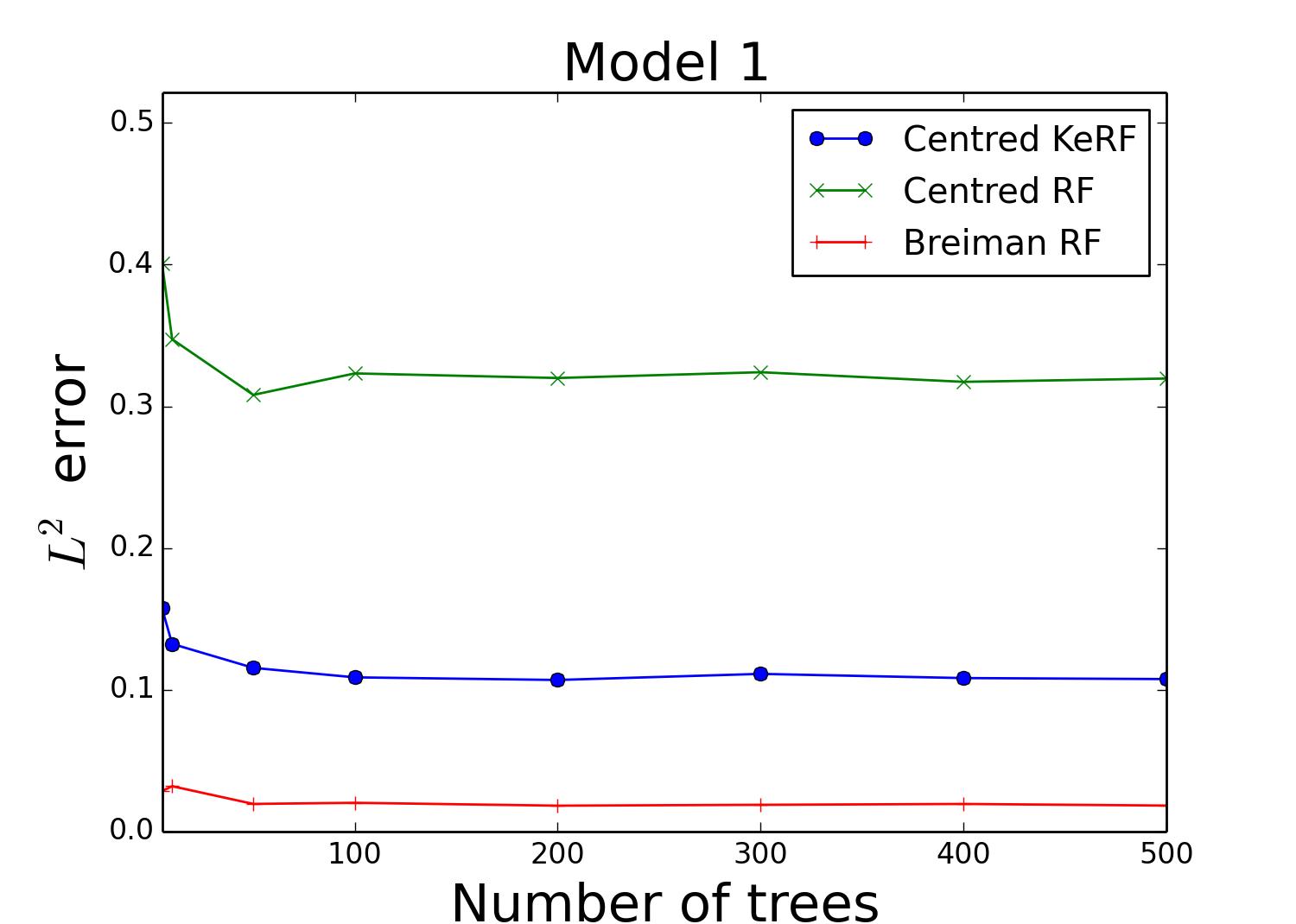}
& \includegraphics[width=0.45\textwidth]{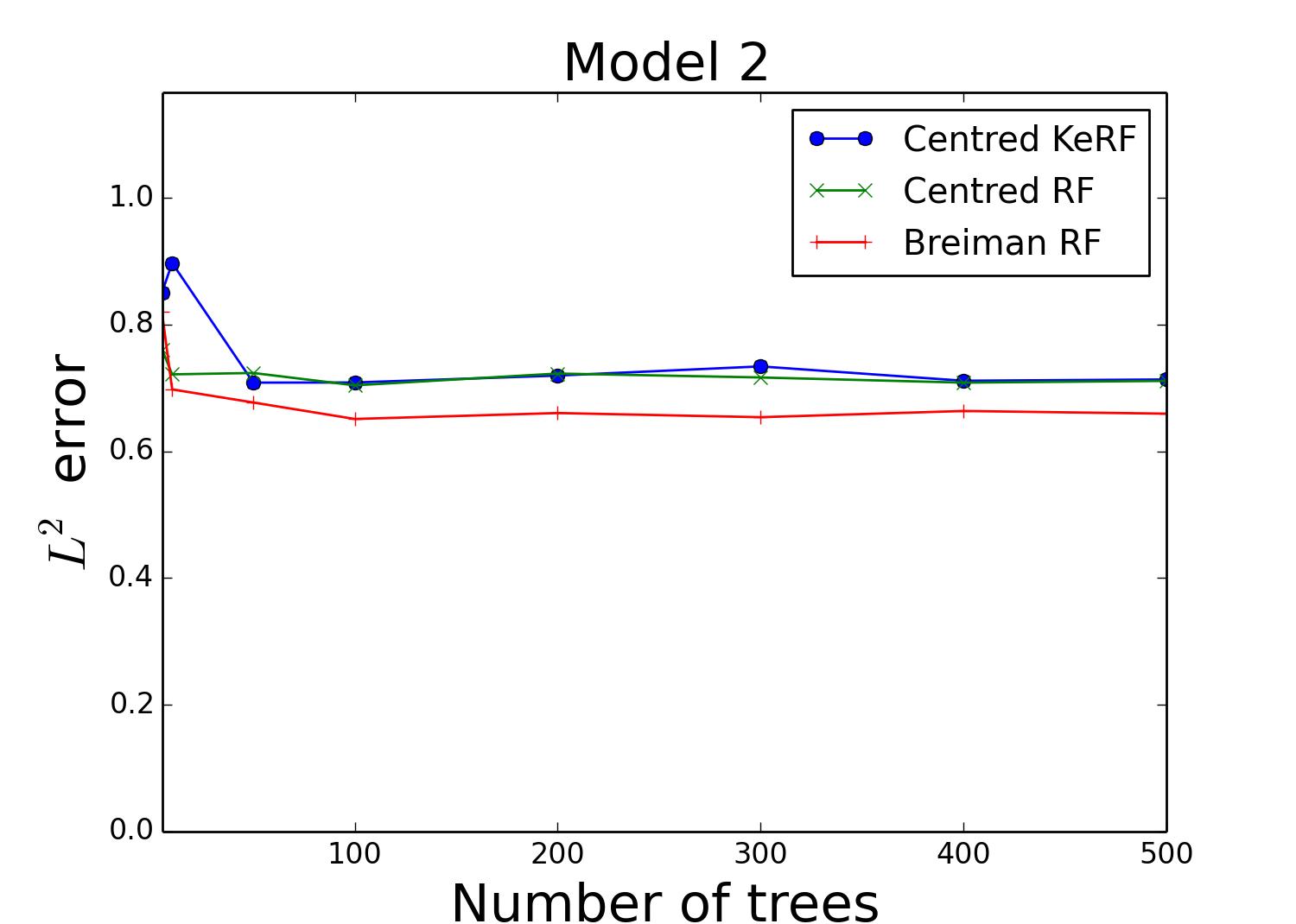}
\end{tabular}
\caption{Empirical risks of centred KeRF and centred forests (with bootstrap).}
\label{figure_7}
\end{center}
\end{figure}

An important aspect of infinite centred and uniform KeRF is that they can be explicitly computed (see Proposition \ref{lemme_centred_random_forest} and \ref{lemme_foret_uniforme_expressiondunoyau}). Thus, we have plotted in Figure \ref{figure_4} the empirical risk of both finite and infinite centred KeRF estimates for some examples (for $n=100$ and $d=10$). We clearly see in this figure that the accuracy of finite centred KeRF tends to the accuracy of infinite centred KeRF as $M$ tends to infinity. This corroborates Proposition \ref{convergenceversK}.

\begin{figure}[h!]
\begin{center}
\begin{tabular}{ccc}
\includegraphics[width=0.45\textwidth]{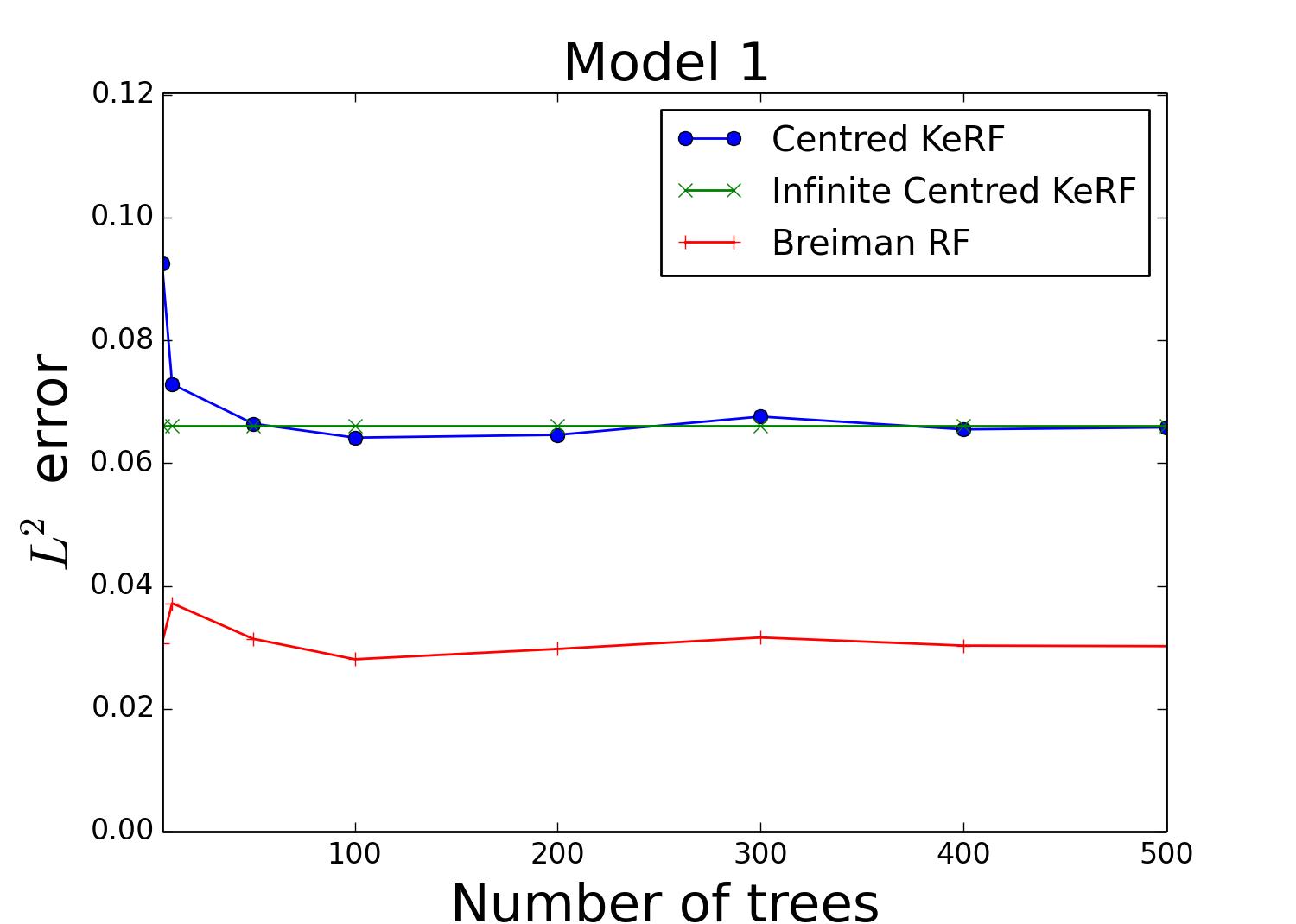}
& \includegraphics[width=0.45\textwidth]{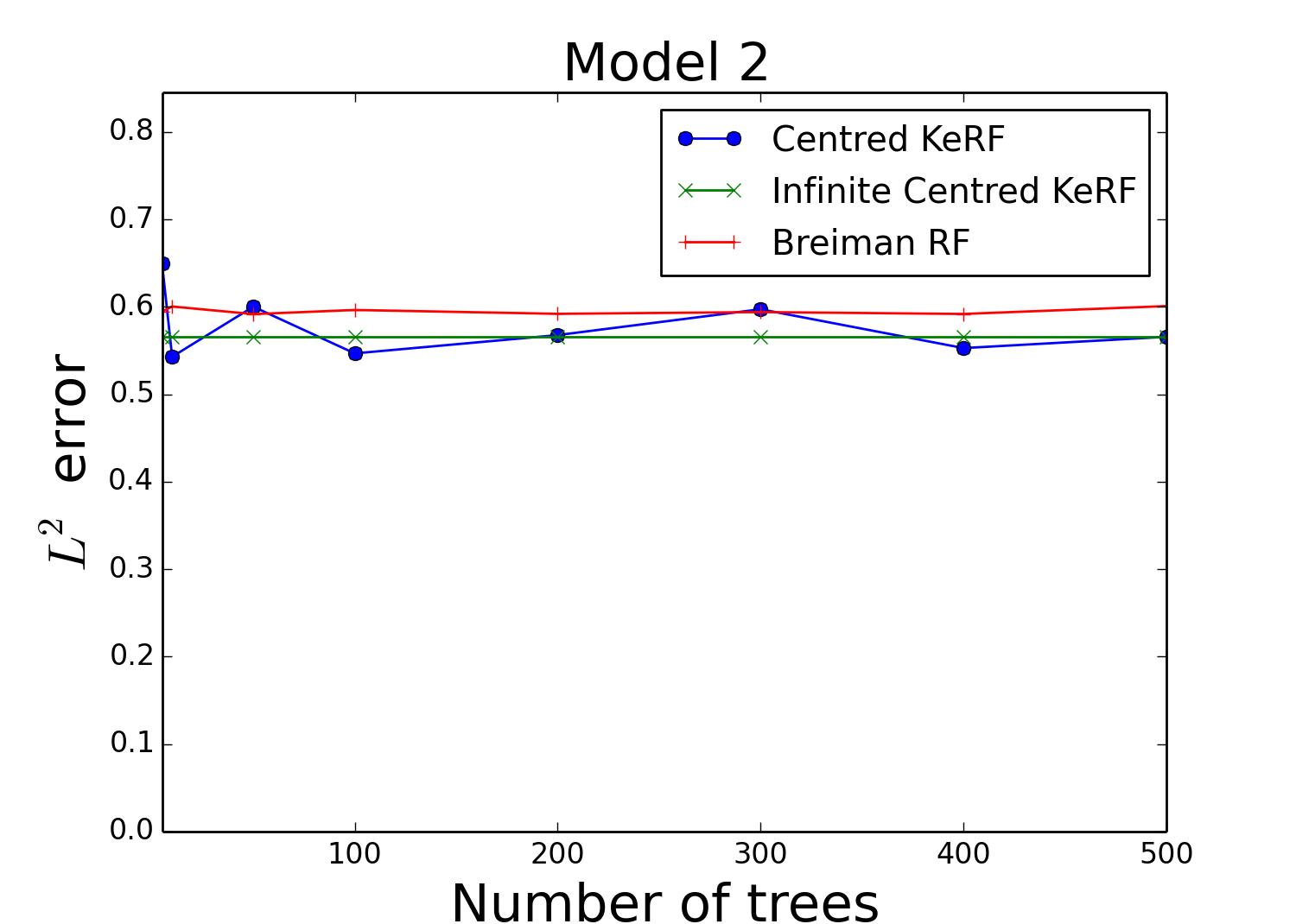}
\end{tabular}
\caption{Risks of finite and infinite centred KeRF.}
\label{figure_4}
\end{center}
\end{figure}

The same comments hold for uniform KeRF (see Figure \ref{figure_5}). Note however that, in that case, the proximity between finite uniform KeRF and infinite uniform KeRF estimate strengthens the approximation that has been made on infinite uniform KeRF in Section $4$.

\begin{figure}[h!]
\begin{center}
\begin{tabular}{ccc}
\includegraphics[width=0.45\textwidth]{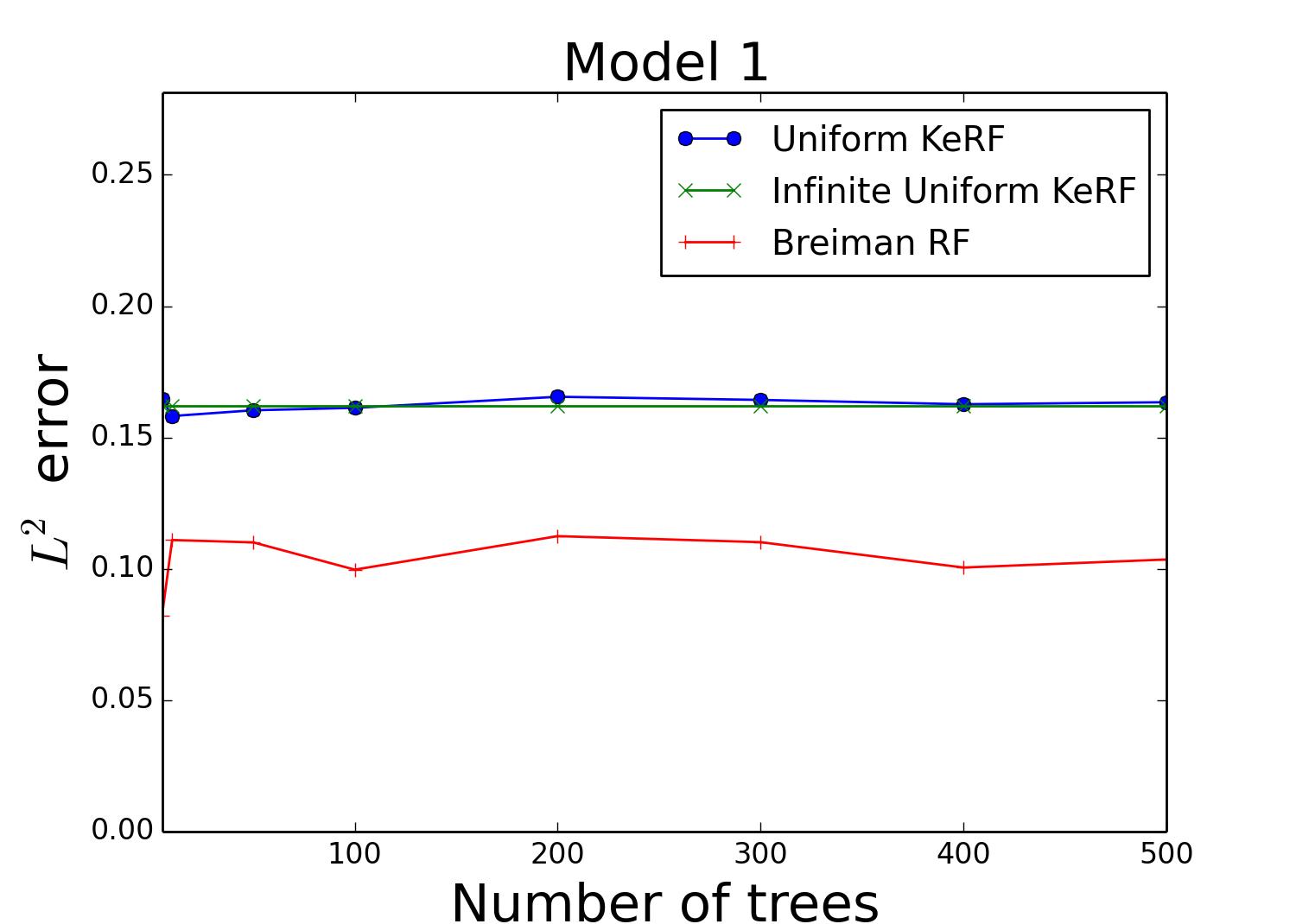}
& \includegraphics[width=0.45\textwidth]{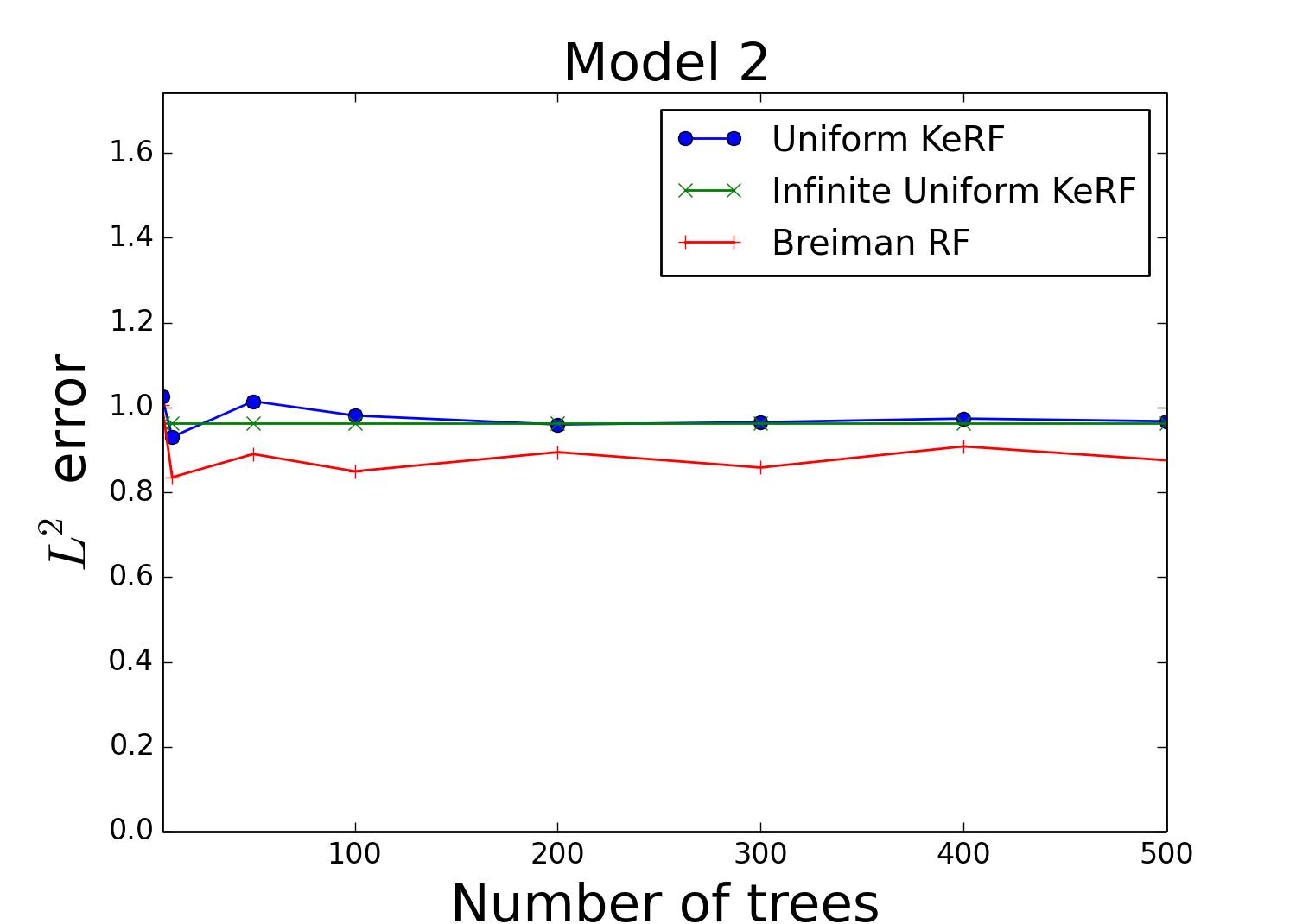}
\end{tabular}
\caption{Risks of finite and infinite uniform KeRF.}
\label{figure_5}
\end{center}
\end{figure}

The computation time for finite KeRF estimate is very acceptable for finite KeRF and similar to that of random forest (Figure \ref{figure_1}-\ref{figure_2}). However, the story is different for infinite KeRF estimates. In fact, KeRF estimates can only be evaluated for  low dimensional data sets and small sample sizes. To see this, just note that the explicit formulation of KeRF involves a multinomial distribution (Proposition \ref{lemme_centred_random_forest} and \ref{lemme_foret_uniforme_expressiondunoyau}). Each evaluation of the multinomial creates computational burden when the dimensions ($d$ and $n$) of the problems increases. For example, in Figure \ref{figure_4} and \ref{figure_5}, the computation time needed to compute infinite KeRF estimates ranges between thirty minutes to $3$ hours. As a matter of fact, infinite KeRF methods should be seen as theoretical tools rather than a practical substitute for random forests.

\section{Proofs}

\begin{proof}[Proof of Proposition \ref{lemme1}]
By definition, 
\begin{align*}
\widetilde{m}_{M,n}(\x, \T_1, \hdots, \T_M ) & =  \frac{1}{\sum_{j=1}^M \sum_{i=1}^n    \mathds{1}_{{\bf X}_i \in A_n(\x, \T_j)}} \sum_{j=1}^M \sum_{i=1}^n Y_i   \mathds{1}_{{\bf X}_i \in A_n(\x, \T_j)} \\
 & =  \frac{M}{\sum_{j=1}^M \sum_{i=1}^n    \mathds{1}_{{\bf X}_i \in A_n(\x, \T_j)}}  \sum_{i=1}^n Y_i   K_{M, n}(\x,\bX_i).
\end{align*}
Finally, observe that
\begin{align*}
\frac{1}{M}\sum_{j=1}^M \sum_{i=1}^n    \mathds{1}_{{\bf X}_i \in A_n(\x, \T_j)} = \sum_{i=1}^n K_{M, n}(\x,\bX_i),
\end{align*}
which concludes the proof.
\end{proof}

\begin{proof}[Proof of Proposition \ref{convergenceversK}]

We prove the result for $d=2$. The other cases can be treated similarly. For the moment, we assume the random forest to be continuous. Recall that, for all $\x,\z\in [0,1]^2$, and for all $M \in \mathds{N}$,  
\begin{align*}
K_{M, n}(\x,\z) = \frac{1}{M} \sum_{j=1}^M \mathds{1}_{ \z \in A_n(\x, \Theta_j)}.
\end{align*}
According to the strong law of large numbers, almost surely, for all $\bx,\bz \in \mathds{Q}^2\cap [0,1]^2$ 
\begin{align*}
\lim\limits_{M \to \infty} K_{M, n}(\x,\z) = K_{n}(\x,\z).
\end{align*}
Set $\varepsilon >0$ and $\bx,\bz \in [0,1]^2$ where $\bx = (x^{(1)},x^{(2)})$ and $\bz = (z^{(1)},z^{(2)})$. Assume, without loss of generality, that  $x^{(1)} < z^{(1)}$ and $x^{(2)} < z^{(2)}$. Let 
\begin{align*}
& A_{\bx} = \{ {\bf u} \in [0,1]^2, u^{(1)} \leq x^{(1)} ~\textrm{and}~ u^{(2)} \leq x^{(2)} \}, \\
\textrm{and}~& A_{\bz} = \{ {\bf u} \in [0,1]^2, u^{(1)} \geq z^{(1)} ~\textrm{and}~ u^{(2)} \geq z^{(2)} \}.
\end{align*}
Choose $\bx_1 \in A_{\bx} \cap \mathds{Q}^2$ (resp. $\bz_2 \in A_{\bz} \cap \mathds{Q}^2$) and take $\bx_2 \in [0,1]^2\cap \mathds{Q}^2$ (resp. $\bz_1 \in [0,1]^2\cap \mathds{Q}^2$) such that $x_1^{(1)} \leq x^{(1)} \leq x_2^{(1)}$ and $x_1^{(2)} \leq x^{(2)} \leq x_2^{(2)}$  (resp. $z_1^{(1)} \leq z^{(1)} \leq z_2^{(1)}$ and $z_1^{(2)} \leq z^{(2)} \leq z_2^{(2)}$, see Figure \ref{figure_1bb}).
\begin{figure}[h!!]
\begin{center}
\includegraphics[scale=0.4]{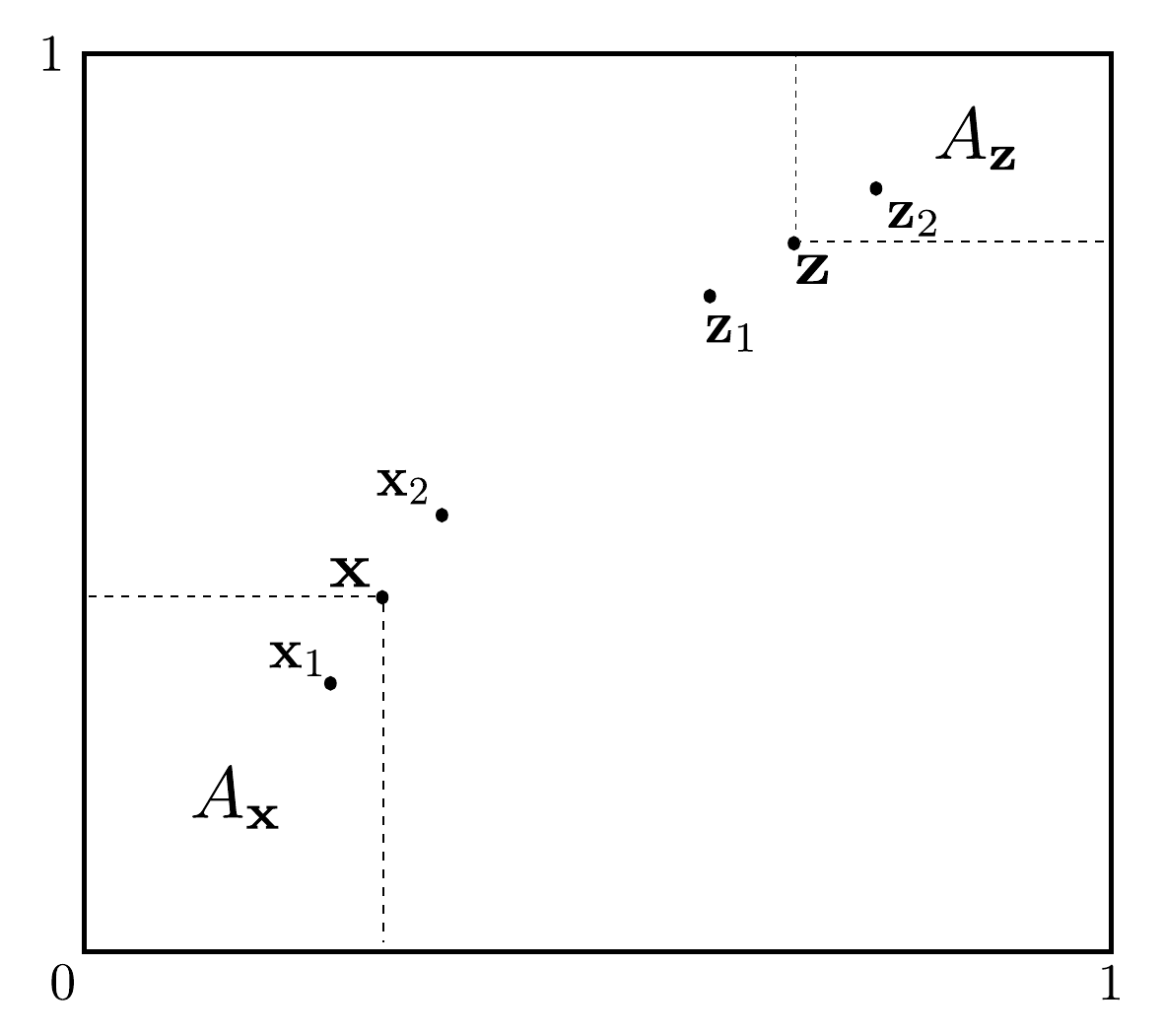}   
\end{center}
\caption{Respective positions of $\bx, \bx_1, \bx_2$ and $\bz, \bz_1, \bz_2$}
\label{figure_1bb}
\end{figure}

Observe that, because of the continuity of $K_n$, one can choose $\bx_1, \bx_2$ close enough to $\x$ and $\bz_2, \bz_1$ close enough to $\z$ such that
\begin{align*}
& |K_{n}(\bx_2, \bx_1)-1| \leq \varepsilon, \\
& |K_{n}(\bz_1, \bz_2)-1| \leq \varepsilon, \\
\textrm{and}~ & |K_{n}(\bx_1,\bz_2) - K_{n}(\x,\z)| \leq \varepsilon .
\end{align*}
Bounding the difference between $K_{M, n}$ and $K_{n}$, we have
\begin{align}
\left| K_{M, n}(\x,\z) - K_{n}(\x,\z) \right| \leq & \left| K_{M, n}(\x,\z) - K_{M, n}(\bx_1,\bz_2) \right| \nonumber \\
& + \left| K_{M, n}(\bx_1,\bz_2) - K_{n}(\bx_1,\bz_2) \right|  \nonumber \\
& + \left| K_{n}(\bx_1,\bz_2) - K_{n}(\x,\z) \right|. \label{lemme1.0}
\end{align}
To simplify notation, we let $\x \overset{\Theta_m}{\leftrightarrow} \z$ be the event where $\bx$ and $\bz$ are in the same cell in the tree built with randomness $\Theta_j$ and dataset $\mathcal{D}_n$. We also let $\x \overset{\Theta_m}{\nleftrightarrow} \z$ be the complement event of $\x \overset{\Theta_m}{\leftrightarrow} \z$. Accordingly, the first term on the right side in equation (\ref{lemme1.0}) is bounded above by
\begin{align}
\left| K_{M, n}(\x,\z) - K_{M, n}(\bx_1,\bz_2) \right| & \leq  \frac{1}{M}\sum_{m=1}^M \left| \mathds{1}_{\x \overset{\Theta_m}{\leftrightarrow} \z} - \mathds{1}_{\bx_1 \overset{\Theta_m}{\leftrightarrow} \bz_2} \right| \nonumber \\
& \leq \frac{1}{M}\sum_{m=1}^M \mathds{1}_{\bx_1 \overset{\Theta_m}{\nleftrightarrow} \x } + \mathds{1}_{\bz_2 \overset{\Theta_m}{\nleftrightarrow} \z }\nonumber\\
& \quad \textrm{(given the positions of $\bx, \bx_1,  \bz, \bz_2$)}\nonumber\\
& \leq \frac{1}{M}\sum_{m=1}^M \mathds{1}_{\bx_1 \overset{\Theta_m}{\nleftrightarrow} \bx_2 } 
+  \frac{1}{M}\sum_{m=1}^M \mathds{1}_{\bz_2 \overset{\Theta_m}{\nleftrightarrow} \bz_1 }, \label{inequality_proof_lemme_3_2}
\end{align}
given the respective positions of $\bx, \bx_1, \bx_2$ and $ \bz, \bz_1, \bz_2$. But, since $\bx_2, \bz_1, \bx_1, \bz_2 \in \mathds{Q}^2\cap [0,1]^2$, we deduce from inequation (\ref{inequality_proof_lemme_3_2}) that, for all $M$ large enough,
\begin{align*}
\left|  K_{M, n}(\bx,\bz) -  K_{M, n}(\bx_1,\bz_2) \right|  \leq 1 - K_{n}(\bx_2, \bx_1) + 1 - K_{n}(\bz_1, \bz_2) + 2 \varepsilon . 
\end{align*}
Combining the last inequality with equation (\ref{lemme1.0}), we obtain, for all $M$ large enough,
\begin{align*}
 \left|  K_{M, n}(\x,\z) - K_{n}(\x,\z) \right| & \leq  ~ 1 - K_{n}(\bx_2, \bx_1) + 1 - K_{n}(\bz_1, \bz_2)\\
 & \quad  + \left|  K_{M, n}(\bx_1,\bz_2) - K_{n}(\bx_1,\bz_2) \right|  \\
 & \quad + \left| K_{n}(\bx_1,\bz_2) - K_{n}(\x,\z) \right| + 2 \varepsilon \\
 & \leq 6 \varepsilon.
\end{align*}
Consequently, for any continuous random forest, almost surely, for all $ \x, \z \in [0,1]^2$,
\begin{align*}
\lim\limits_{M \to \infty} K_{M, n}(\x,\z) = K_{n}(\x,\z).
\end{align*}
The proof can be easily adapted to the case of discrete random forests. Thus, this complete the first part of the proof. Next, observe that
\begin{align*}
\lim\limits_{M \to \infty }\frac{\sum_{i=1}^n Y_i K_{M, n}(\x,\X_i)}{\sum_{j=1}^n K_{M, n}(\x,\X_j)} = \frac{\sum_{i=1}^n Y_i K_{n}(\x,\X_i)}{\sum_{j=1}^n K_{n}(\x,\X_j)}, 
\end{align*}
for all $\bx$ satisfying $\sum_{j=1}^n K_{n}(\x,\X_j) \neq 0$. Thus, almost surely for those $\bx$, 
\begin{align}
\lim\limits_{M \to \infty } \widetilde{m}_{M,n}(\x) = \widetilde{m}_{\infty,n}(\x). \label{proof_resume_formula}
\end{align}
Now, if there exists any $\bx$ such that $\sum_{j=1}^n K_{n}(\x,\X_j) = 0$, then $\bx$ is not connected with any data points in any tree of the forest. In that case, $\sum_{j=1}^n$ $ K_{M, n}(\x,\X_j)$ $ = 0$ and, by convention, $\widetilde{m}_{\infty,n}(\x) = \widetilde{m}_{M,n}(\x) = 0$. Finally, formula (\ref{proof_resume_formula}) holds for all $\x \in [0,1]^2$.
\end{proof}

\begin{proof}[Proof of Proposition \ref{lemme2}]
Fix $\bx \in [0,1]^d$ and assume that, a.s., $Y\geq 0$. By assumption {\bf (H1.1)}, there exist sequences $(a_n), (b_n)$ such that, almost surely,
\begin{align*}
a_n \leq N_n(\x, \T) \leq b_n.
\end{align*}
To simplify notation, we let $\bar{N}_{M,n}(\x, \Theta)=\frac{1}{M}\sum_{j=1}^M$ $ N_n(\x, \T_j)$. Thus, almost surely, 
\begin{align*}
|m_{M,n }({\bf x}) - \widetilde{m}_{M,n}(\bx)| & = \bigg| \sum_{i=1}^n Y_i \left( \frac{1}{M} \sum_{m=1}^M \frac{\mathds{1}_{\X_i \in A_n(\x, \Theta_m)}}{N_n(\bx, \Theta_m)} \right) \\
& \qquad - \sum_{i=1}^n Y_i \left( \frac{1}{M} \sum_{m=1}^M \frac{\mathds{1}_{\X_i \in A_n(\x, \Theta_m)}}{\bar{N}_{M,n}(\x)}    \right) \bigg|\\
& \leq  \frac{1}{M} \sum_{i=1}^n Y_i   \sum_{m=1}^M \frac{\mathds{1}_{\X_i \in A_n(\x, \Theta_m)}}{\bar{N}_{M,n}(\x)} \times \bigg| \frac{\bar{N}_{M,n}(\x)}{N_n(\bx, \Theta_m)} - 1 \bigg|   \\
& \leq   \frac{b_n-a_n}{a_n} \widetilde{m}_{M,n}(\bx).
\end{align*}
\end{proof}

\begin{proof}[Proof of Proposition \ref{lemme5}]

Fix $\bx \in [0,1]^d$ and assume that, almost surely, $Y\geq 0$. By assumption {\bf (H1.2)}, there exist sequences $(a_n)$, $(b_n)$, $(\varepsilon_n)$ such that, letting $A$ be the event where 
\begin{align*}
a_n \leq N_n(\x, \T) \leq b_n,
\end{align*}
we have, almost surely, 
\begin{align*}
\P_{\Theta}[A] \geq 1 - \varepsilon_n \quad \textrm{and} \quad  1 \leq a_n \leq \E_{\T} \left[N_n(\x, \Theta) \right]  \leq b_n.
\end{align*}
Therefore, a.s., 
\begin{align*}
& |m_{\infty,n }({\bf x}) - \widetilde{m}_{\infty, n}(\x)| \\
& = \bigg| \sum_{i=1}^n Y_i \mathds{E}_{\Theta} \left[ \frac{\mathds{1}_{\X_i \in A_n(\x, \Theta)}}{N_n(\x, \Theta)} \right] - \sum_{i=1}^n Y_i \mathds{E}_{\Theta} \left[ \frac{\mathds{1}_{\X_i \in A_n(\x, \Theta)}}{\E_{\T} \left[N_n(\x, \Theta) \right]} \right] \bigg|\\
& = \bigg| \sum_{i=1}^n Y_i \mathds{E}_{\Theta} \bigg[ \frac{\mathds{1}_{\X_i \in A_n(\x, \Theta)}}{\E_{\T} \left[N_n(\x, \Theta) \right]} 
\bigg( \frac{\E_{\T} \left[N_n(\x, \Theta) \right]}{N_n(\x, \Theta)}- 1 \bigg) (\mathds{1}_{A} + \mathds{1}_{A^c} ) \bigg] \bigg|\\
& \leq  \frac{b_n - a_n}{a_n } \widetilde{m}_{\infty, n}(\x) + \Big(\max_{1 \leq i \leq n} Y_i \Big)   \mathds{E}_{\Theta} \bigg[ \Big| 1 - \frac{N_n(\x, \Theta)}{\E_{\T} \left[N_n(\x, \Theta) \right]} 
\Big| \mathds{1}_{A^c} \bigg] \\
& \leq  \frac{b_n - a_n}{a_n } \widetilde{m}_{\infty, n}(\x) + n \Big(\max_{1 \leq i \leq n} Y_i \Big) \mathds{P}[A^c].
\end{align*}
Consequently, almost surely, 
\begin{align*}
|m_{\infty,n }({\bf x}) - \widetilde{m}_{\infty, n}(\x)| & \leq  \frac{b_n - a_n}{a_n } \widetilde{m}_{\infty, n}(\x) + n \varepsilon_n \Big(\max_{1 \leq i \leq n} Y_i \Big).
\end{align*}

\end{proof}

\begin{proof}[Proof of Proposition \ref{lemme_centred_random_forest}]
Assume for the moment that $d=1$. Take $x,z \in [0,1]$ and assume, without loss of generality, that $x\leq z$. Then the probability that $x$ and $z$ be in the same cell, after $k$ cuts, is equal to  
\begin{align*}
\mathds{1}_{ \lceil 2^{k_j}x_j \rceil = \lceil 2^{k_j}z_j \rceil}.
\end{align*}

To prove the result in the multivariate case, take $\bx, \bz \in [0,1]^d$. Since cuts are independent, the probability that $\bx$ and $\bz$ are in the same cell after $k$ cuts is given by the following multinomial 
\begin{align*}
K_{k}^{cc}(\bx, \bz) = \sum\limits_{\substack{k_1,\hdots,k_d \\ \sum_{\ell=1}^d k_{\ell} = k }} 
\frac{k!}{k_1! \hdots k_d !} \prod_{j=1}^d \bigg( \frac{1}{d}\bigg)^{k_j} \mathds{1}_{ \lceil 2^{k_j}x_j \rceil = \lceil 2^{k_j}z_j \rceil}.
\end{align*}
\end{proof}

To prove Theorem \ref{theoreme_consistency_centred_forest_approximation}, we need to control the bias of the centred KeRF estimate, which is done in Theorem \ref{bias_theorem_centred_forest}. 

\begin{theorem}\label{bias_theorem_centred_forest}
Let $f$ be a $L$-Lipschitz function. Then, for all $k$,
\begin{align*}
\sup\limits_{\bx \in [0,1]^d} \left| \frac{\int_{[0,1]^d} K^{cc}_{k}(\x, \z) f(\z) \diff z_1 \hdots \diff z_d }{\int_{[0,1]^d} K^{cc}_{k}(\x, \z) \diff z_1 \hdots \diff z_d} - f(\x) \right| \leq Ld \left(1 - \frac{1}{2d}\right)^k.
\end{align*}
\end{theorem}

\begin{proof}[Proof of Theorem \ref{bias_theorem_centred_forest}]

Let $\bx \in [0,1]^d$ and $k \in \mathds{N}$. Take $f$ a $L$-Lipschitz function. In the rest of the proof, for clarity reasons, we use the notation $\diff \bz$ instead of $\diff z_1 \hdots \diff z_d$. Thus, 
\begin{align*}
  \left| \frac{\int_{[0,1]^d} K^{cc}_k(\x, \z) f(\z) \diff \bz }{\int_{[0,1]^d} K^{cc}_k(\x, \z) \diff \bz} - f(\x) \right| 
   \leq   \frac{\int_{[0,1]^d} K^{cc}_k(\x, \z) |f(\z)-f(\bx)| \diff \bz }{\int_{[0,1]^d} K^{cc}_k(\x, \z) \diff \bz}.
\end{align*}
Note that,
\begin{align}
& \int_{[0,1]^d} K^{cc}_k(\x, \z) |f(\z)-f(\bx)| \diff \bz \nonumber\\
& \quad \leq L \sum_{\ell = 1}^d \int_{[0,1]^d} K^{cc}_k(\x, \z) |z_{\ell}-x_{\ell}| \diff \bz \nonumber\\
& \quad \leq L \sum_{\ell = 1}^d \int_{[0,1]^d} \sum\limits_{\substack{k_1,\hdots,k_d \\ \sum_{j=1}^d k_j = k }} \frac{k!}{k_1! \hdots k_d !}  \left( \frac{1}{d}\right)^k\prod_{m\neq \ell} \int_0^1 K^{cc}_{k_m}(x_m,z_m) \diff z_m \nonumber\\
& \qquad \qquad  \qquad \times  \int_0^1 K^{cc}_{k_{\ell}}(x_{\ell},z_{\ell})|z_{\ell}-x_{\ell}| \diff z_{\ell}.\label{equation_1_proof_bias_centred}
\end{align}
The last integral is upper bounded by
\begin{align}
\int_{[0,1]}   K^{cc}_{k_{\ell}}(x_{\ell},z_{\ell}) |x_{\ell} - z_{\ell}| \textrm{d}z_{\ell} & = \int_{[0,1]}   \mathds{1}_{ \lceil 2^{k_{\ell}}x_{\ell} \rceil = \lceil 2^{k_{\ell}}z_{\ell} \rceil} |x_{\ell} - z_{\ell}| \textrm{d}z_{\ell} \nonumber \\
& \leq \left( \frac{1}{2}\right)^{k_{\ell}} \int_{[0,1]}   \mathds{1}_{ \lceil 2^{k_{\ell}}x_{\ell} \rceil = \lceil 2^{k_{\ell}}z_{\ell} \rceil} \textrm{d}z_{\ell} \nonumber \\
& \leq \left( \frac{1}{2}\right)^{k_{\ell}} \int_{[0,1]}   K^{cc}_{k_{\ell}}(x_{\ell},z_{\ell}) \textrm{d}z_{\ell}.\label{equation_1b_proof_bias_centred}
\end{align}
Therefore, combining inequalities (\ref{equation_1_proof_bias_centred}) and (\ref{equation_1b_proof_bias_centred}), we obtain,
\begin{align}
& \int_{[0,1]^d} K^{cc}_k(\x, \z) |f(\z)-f(\bx)| \diff \bz \nonumber\\
 & \quad \leq L \sum_{\ell = 1}^d \sum\limits_{\substack{k_1,\hdots,k_d \\ \sum_{j=1}^d k_j = k }} \frac{k!}{k_1! \hdots k_d !} \left(\frac{1}{2}\right)^{k_{\ell}} \left(\frac{1}{d}\right)^{k} \prod_{m = 1}^d  \int_0^1 K^{cc}_{k_m}(x_m,z_m) \diff z_m \nonumber\\
  & \quad  \leq L  \left( \frac{1}{d}\right)^k \sum_{\ell = 1}^d   \sum\limits_{\substack{k_1,\hdots,k_d \\ \sum_{j=1}^d k_j = k }} \frac{k!}{k_1! \hdots k_d !} \left(\frac{1}{2}\right)^{k_{\ell}+k},\label{equation_2_proof_bias_centred}
\end{align}
since, simple calculations show that, for all $x_m \in [0,1]$ and for all $k_m \in \mathds{N}$,  
 \begin{align}
 \int_0^1 K^{cc}_{k_m}(x_m,z_m) \diff z_m = \int_{[0,1]}   \mathds{1}_{ \lceil 2^{k_m}x_m \rceil = \lceil 2^{k_m}z_m \rceil}  dz_m = \left( \frac{1}{2} \right)^{k_m}.  \label{fact1}
\end{align}
Consequently, we get from inequality (\ref{equation_2_proof_bias_centred}) that
\begin{align*}
& \frac{\int_{[0,1]^d} K^{cc}_k(\x, \z) |f(\z)-f(\bx)| \diff \bz }{\int_{[0,1]^d} K^{cc}_k(\x, \z) \diff \bz} 
 \leq L \left( \frac{1}{d}\right)^k \sum_{\ell = 1}^d   \sum\limits_{\substack{k_1,\hdots,k_d \\ \sum_{j=1}^d k_j = k }} \frac{k!}{k_1! \hdots k_d !} \left(\frac{1}{2}\right)^{k_{\ell}}.
\end{align*}
Taking the first term of the sum, we obtain
\begin{align*}
\left(\frac{1}{d}\right)^k \sum\limits_{\substack{k_1,\hdots,k_d \\ \sum_{j=1}^d k_j = k }} \frac{k!}{k_1! \hdots k_d !} \left(\frac{1}{2}\right)^{k_1} 
 &  =  \sum_{k_1 = 0}^k  \left( \frac{1}{2d} \right)^{k_1}  \left(1-\frac{1}{d}\right)^{k-k_1} \frac{k!}{k_1! (k-k_1)!} \\
 &  \leq \left(1 - \frac{1}{2d}\right)^k.
\end{align*}
Finally, 
\begin{align*}
& \frac{\int_{[0,1]^d} K^{cc}_k(\x, \z) |f(\z)-f(\bx)| \diff \bz }{\int_{[0,1]^d} K^{cc}_k(\x, \z) \diff \bz} \leq Ld \left(1 - \frac{1}{2d}\right)^k.
\end{align*}
\end{proof}

\begin{proof}[Proof of Theorem \ref{theoreme_consistency_centred_forest_approximation}]

Let $\bx \in [0,1]^d$, $\|m\|_{\infty} = \sup\limits_{\bx \in [0,1]^d} |m(\bx)|$ and recall that
\begin{align*}
\widetilde{m}_{\infty, n}^{cc}(\x) = & \frac{\sum_{i=1}^n Y_i K_k^{cc}(\x, \X_i)}{\sum_{i=1}^n  K_k^{cc}(\x, \X_i)}.
\end{align*}
Thus, letting 
\begin{align*}
& A_{n}(\x) = \frac{1}{n}\sum_{i=1}^n \left( \frac{Y_i K_k^{cc}(\x, \X_i)}{\E \left[ K_k^{cc}(\x, \X) \right]} - \frac{\E \left[ YK_k^{cc}(\x, \X) \right]}{\E \left[ K_k^{cc}(\x, \X) \right]} \right),\\
& B_{n}(\x) = \frac{1}{n} \sum_{i=1}^n \left( \frac{K_k^{cc}(\x, \X_i)}{\E \left[ K_k^{cc}(\x, \X) \right]} - 1 \right),\\
\textrm{and}~ & M_n(\x) = \frac{\E \left[ YK_k^{cc}(\x, \X) \right]}{\E \left[ K_k^{cc}(\x, \X) \right]}, 
\end{align*}
the estimate $\widetilde{m}_{\infty, n}^{cc}(\x)$ can be rewritten as 
\begin{align*}
\widetilde{m}_{\infty, n}^{cc}(\x) = \frac{M_n(\x) + A_n(\x)}{1 + B_n(\x)},
\end{align*}
which leads to 
\begin{align*}
\widetilde{m}_{\infty, n}^{cc}(\x) - m(\x)& = \frac{  M_n(\x) - m(\x) + A_n(\x) - B_n(\x) m(\x)}{1 + B_n(\x)}.
\end{align*}
According to Theorem \ref{bias_theorem_centred_forest}, we have
\begin{align*}
|M_n(\x) - m(\x)| & = \left|\frac{\E \left[ m(\bX) K_k^{cc}(\x, \X) \right]}{\E \left[ K_k^{cc}(\x, \X) \right]} + \frac{\E \left[ \varepsilon K_k^{cc}(\x, \X) \right]}{\E \left[ K_k^{cc}(\x, \X) \right]} - m(\bx) \right|\\
& \leq  \left|\frac{\E \left[ m(\bX) K_k^{cc}(\x, \X) \right]}{\E \left[ K_k^{cc}(\x, \X) \right]} - m(\bx) \right|\\
& \leq C_1 \left( 1 - \frac{1}{2d}\right)^k, 
\end{align*}
where $C_1=Ld$. Take $\alpha\in ]0, 1/2]$. Let $\mathcal{C}_{\alpha}(\bx)$ be the event on which $\big\lbrace |A_n(\x)|,$ $ |B_n(\x)| \leq \alpha \big\rbrace$. On the event $\mathcal{C}_{\alpha}(\bx)$, we have
\begin{align*}
|\widetilde{m}_{\infty, n}^{cc}(\x) - m(\x) |^2 & \leq 8 |M_n(\x) - m(\x) |^2 + 8 |A_n(\bx) - B_n(\bx) m(\x) |^2 \nonumber\\
& \leq 8C_1^2 \left( 1 - \frac{1}{2d} \right)^{2k} + 8\alpha^2 (1 + \|m\|_{\infty})^2.
\end{align*}
Thus, 
\begin{align}
\E [ |\widetilde{m}_{\infty, n}^{cc}(\x) - m(\x)|^2 \mathds{1}_{\mathcal{C}_{\alpha}(\bx)} ] & \leq 8C_1^2 \left( 1 - \frac{1}{2d} \right)^{2k} + 8\alpha^2 (1 + \|m\|_{\infty})^2.\label{equation_proof_rate_consistency_centred}
\end{align}
Consequently, to find an upper bound on the rate of consistency of $\widetilde{m}_{\infty, n}^{cc}$, we just need to upper bound 
\begin{align*}
\E \Big[ |\widetilde{m}_{\infty, n}^{cc}(\x) - m(\x)|^2 \mathds{1}_{\mathcal{C}^c_{\alpha}(\bx)} \Big] 
& \leq \E \Big[ |\max\limits_{1 \leq i \leq n} Y_i  + m(\x)|^2 \mathds{1}_{\mathcal{C}^c_{\alpha}(\bx)} \Big] \nonumber \\
& \quad \textrm{(since  $\widetilde{m}_{\infty, n}^{cc}$ is a local averaging estimate)}\nonumber\\
& \leq \E \Big[ |2 \|m\|_{\infty} + \max\limits_{1 \leq i \leq n} \varepsilon_i |^2 \mathds{1}_{\mathcal{C}^c_{\alpha}(\bx)} \Big] \nonumber\\
& \leq \left( \E \left[ 2\|m\|_{\infty} + \max\limits_{1 \leq i \leq n} \varepsilon_i\right]^4  \P \left[ \mathcal{C}^c_{\alpha}(\bx) \right] \right)^{1/2} \nonumber \\
& \quad \textrm{(by Cauchy-Schwarz inequality)} \nonumber \\
& \leq \left(  \left( 16 \|m\|_{\infty}^4 + 8 \E \Big[ \max\limits_{1 \leq i \leq n} \varepsilon_i\Big]^4 \right) \P \left[ \mathcal{C}^c_{\alpha}(\bx) \right] \right)^{1/2}. \nonumber
\end{align*}
Simple calculations on Gaussian tails show that one can find a constant $C>0$ such that for all $n$, 
\begin{align*}
\E \Big[ \max\limits_{1 \leq i \leq n} \varepsilon_i \Big]^4 \leq C (\log n)^2.
\end{align*}
Thus, there exists $C_2$ such that, for all $n >1$, 
\begin{align}
\E \Big[ |\widetilde{m}_{\infty, n}^{cc}(\x) - m(\x)|^2 \mathds{1}_{\mathcal{C}^c_{\alpha}(\bx)} \Big] 
& \leq C_2 (\log n) (\P \left[ \mathcal{C}^c_{\alpha}(\bx) \right] )^{1/2}. \label{proba_centred_ineq}
\end{align}
The last probability $\P \left[ \mathcal{C}^c_{\alpha}(\bx) \right] $ can be upper bounded by using Chebyshev's inequality. Indeed, with respect to $A_n(\bx)$, 
\begin{align*}
\P \big[ |A_n(\bx)| > \alpha \big] 
& \leq \frac{1}{n \alpha^2 }\E \bigg[ \frac{Y K_k^{cc}(\x, \X)}{\E \left[ K_k^{cc}(\x, \X) \right]} - \frac{\E \left[ YK_k^{cc}(\x, \X) \right]}{\E \left[ K_k^{cc}(\x, \X) \right]} \bigg]^2\\
& \leq \frac{1}{n \alpha^2 } \frac{1}{(\E \left[ K_k^{cc}(\x, \X) \right])^2}\E \bigg[ Y^2 K_k^{cc}(\x, \X)^2 \bigg]\\
& \leq \frac{2}{n \alpha^2 } \frac{1}{(\E \left[ K_k^{cc}(\x, \X) \right])^2}\bigg( \E \bigg[ m(\bX)^2 K_k^{cc}(\x, \X)^2 \bigg] \\
& \qquad +\E \bigg[ \varepsilon^2 K_k^{cc}(\x, \X)^2 \bigg]\bigg)\\
&\leq \frac{2(\|m\|_{\infty}^2+\sigma^2)}{n \alpha^2 } \frac{\E \left[ K_k^{cc}(\x, \X) \right]}{(\E \left[ K_k^{cc}(\x, \X) \right])^2}  \\
& \quad \textrm{(since $\sup\limits_{\bx, \bz \in [0,1]^d}K_k^{cc}(\bx, \bz)\leq 1$)}\\
& \leq \frac{2M_1^2}{ \alpha^2 } \frac{2^k}{n}\\
& \quad \textrm{(according to inequality (\ref{fact1}))},
\end{align*}
where $M_1^2 = \|m\|_{\infty}^2+\sigma^2$. Meanwhile with respect to $B_n(\bx)$, we obtain, still by Chebyshev's inequality, 
\begin{align*}
\P \big[ |B_n(\bx)| > \alpha \big] & \leq \frac{1}{n \alpha^2} \E \bigg[ \frac{K_k^{cc}(\x, \X_i)}{\E \left[ K_k^{cc}(\x, \X) \right]}\bigg]^2\\
& \leq \frac{1}{n \alpha^2} \frac{1}{\E \left[ K_k^{cc}(\x, \X) \right]}\\
& \quad \textrm{(since $\sup\limits_{\bx, \bz \in [0,1]^d}K_k^{cc}(\bx, \bz)\leq 1$)}\\
& \leq \frac{2^{k}}{n \alpha^2}.
\end{align*}
Thus, the probability of $\mathcal{C}_{\alpha}(\bx)$ is given by
\begin{align*}
\P \big[\mathcal{C}_{\alpha}(\bx) \big] &  \geq 1 - \P \big(|A_n(\x)| \geq \alpha \big) - \P \big(|B_n(\x)| \geq \alpha \big) \nonumber\\
& \geq 1 -  \frac{2^{k}}{n} \frac{2M_1^2 }{\alpha^2}  - \frac{2^{k}}{n \alpha^2} \nonumber\\
& \geq 1 - \frac{2^{k}(2M_1^2+1)}{n \alpha^2 }. 
\end{align*}
Consequently, according to inequality (\ref{proba_centred_ineq}), we obtain
\begin{align*}
\E \Big[ |\widetilde{m}_{\infty, n}^{cc}(\x) - m(\x)|^2 \mathds{1}_{\mathcal{C}^c_{\alpha}(\bx)} \Big] 
& \leq C_2 (\log n) \Big(\frac{2^{k}(2M_1^2+1)}{n \alpha^2 } \Big)^{1/2}. 
\end{align*}
Then using inequality (\ref{equation_proof_rate_consistency_centred}), 
\begin{align*}
& \E \Big[ \widetilde{m}_{\infty, n}^{cc}(\x) - m(\x) \Big]^2 \\
& \leq \E \Big[ |\widetilde{m}_{\infty, n}^{cc}(\x) - m(\x)|^2 \mathds{1}_{\mathcal{C}_{\alpha}(\bx)} \Big] + \E \Big[ |\widetilde{m}_{\infty, n}^{cc}(\x) - m(\x)|^2 \mathds{1}_{\mathcal{C}^c_{\alpha}(\bx)} \Big] \\
& \leq 8C_1^2 \left( 1 - \frac{1}{2d} \right)^{2k} + 8\alpha^2 (1 + \|m\|_{\infty})^2
+ C_2 (\log n) \Big( \frac{2^{k}(2M_1^2+1)}{n \alpha^2 } \Big)^{1/2}
\end{align*}
Optimizing the right hand side in $\alpha$, we get
\begin{align*}
& \E \Big[ \widetilde{m}_{\infty, n}^{cc}(\x) - m(\x) \Big]^2  \leq 8C_1^2 \left( 1 - \frac{1}{2d} \right)^{2k} 
+ C_3 \bigg( \frac{(\log n)^2 2^k }{n} \bigg)^{1/3},
\end{align*}
for some constant $C_3>0$. The last expression is minimized for 
\begin{align*}
k = C_4 + \frac{1}{{\log2 + \frac{3}{d}}} \log \left( \frac{n}{(\log n)^2}\right),
\end{align*}
where $C_4 = \Big( \frac{1}{d} + \frac{\log 2}{3} \Big)^{-1} \log \left( \frac{C_3 d \log 2}{24 C_1^2} \right)$. Consequently,  there exists a constant $C_5$ such that, for all $n>1$, 
\begin{align*}
& \E \left[ \widetilde{m}_{\infty, n}^{cc}(\x) - m(\x)\right]^2 \leq C_5 n^{- \frac{1}{d \log 2 + 3}} (\log n)^2.
\end{align*}

\end{proof}

\begin{proof}[Proof of Lemma \ref{lemme_inutile}]
Let $x,z \in [0,1]$ such that $x<z$. The first statement comes from the fact that splits are drawn uniformly over $[0,1]$. To address the second one, denote by $Z_1$ (resp. $Z_2$) the position of the first (resp. second) split used to build the cell containing $[x,z]$. Observe that, given $Z_1 = z_1$, $Z_2$ is uniformly distributed over $[z_1,1]$ (resp. $[0,z_1]$) if $z_1\leq x$ (resp. $z_1 \geq z$). Thus, we have
\begin{align*}
K_2^{uf}(x,z) & = \int_{z_1=0}^x \left( \int_{z_2= z_1}^x \frac{1}{1-z_1} \diff z_1 \diff z_2 + \int_{z_2=z}^1 \frac{1}{1-z_1} \diff z_1 \diff z_2 \right)\\
& \quad + \int_{z_1=z}^1 \left( \int_{z_2=0}^x \frac{1}{1-z_1} \diff z_1 \diff z_2 + \int_{z_2=z}^{z_1} \frac{1}{1-z_1} \diff z_1 \diff z_2 \right).
\end{align*} 
The first term takes the form
\begin{align*}
\int_0^x \frac{1}{z_1} \left( \int_{z_1}^x \diff z_2 \right) \diff z_1 & = \int_0^x \frac{x-z_1}{1-z_1} \diff z_1\\
& = x -(1-x) \log (1-x).
\end{align*}
Similarly, one has
\begin{align*}
\int_0^x \int_z^1 \frac{1}{1-z_1}\diff z_1 \diff z_2 & = (1-z) \log (1-x),\\
\int_z^1 \int_z^{z_1} \frac{1}{z_1}\diff z_1 \diff z_2 & = (1-z) + z \log z,\\
\int_z^1 \int_0^{x} \frac{1}{z_1}\diff z_1 \diff z_2 & = - x \log z.
\end{align*}
Consequently, 
\begin{align*}
K_2^{uf}(x,z) & = x -(1-x) \log (1-x) + (1-z) \log (1-x) \\
& \quad - x \log z  + (1-z) + z \log z\\
& = 1 - (z-x) + (z-x) \log \left( \frac{z}{1-x} \right).
\end{align*}

\end{proof}

\begin{proof}[Proof of Proposition \ref{lemme_foret_uniforme_expressiondunoyau}]
\begin{sloppypar} 
The result is proved in Technical Proposition $2$ in \citet{Sc14}.
\end{sloppypar}
\end{proof}

To prove Theorem \ref{theoreme_consistency_uniform_forest_approximation}, we need to control the bias of uniform KeRF estimates, which is done in Theorem \ref{bias_theorem}.

\begin{theorem}\label{bias_theorem}
Let $f$ be a $L$-Lipschitz function. Then, for all $k$,
\begin{align*}
\sup\limits_{\bx \in [0,1]^d} \left| \frac{\int_{[0,1]^d} K^{uf}_{k}(\0, |\z-\x|) f(\z) \diff z_1 \hdots \diff z_d }{\int_{[0,1]^d} K^{uf}_{k}(\0, |\z-\x|) \diff z_1 \hdots \diff z_d} - f(\x) \right| \leq \frac{L d 2^{2d+1}}{3}  \left(1 - \frac{1}{3d}\right)^k.
\end{align*}
\end{theorem}

\begin{proof}[Proof of Theorem \ref{bias_theorem}]

Let $\bx \in [0,1]^d$ and $k \in \mathds{N}$. Let $f$ be a $L$-Lipschitz function. In the rest of the proof, for clarity reasons, we use the notation $\diff \bz$ instead of $\diff z_1 \hdots \diff z_d$. Thus, 
\begin{align*}
  \left| \frac{\int_{[0,1]^d} K_k^{uf}(\0, |\z-\x|) f(\z) \diff \bz }{\int_{[0,1]^d} K_k^{uf}(\0, |\z-\x|) \diff \bz} - f(\x) \right| 
   \leq   \frac{\int_{[0,1]^d} K_k^{uf}(\0, |\z-\x|) |f(\z)-f(\bx)| \diff \bz }{\int_{[0,1]^d} K_k^{uf}(\0, |\z-\x|) \diff \bz}.
\end{align*}
Note that, 
\begin{align}
& \int_{[0,1]^d} K_k^{uf}(\0, |\z-\x|) |f(\z)-f(\bx)| \diff \bz \nonumber\\
& \quad \leq L \sum_{\ell = 1}^d \int_{[0,1]^d} K_k^{uf}(\0, |\z-\x|) |z_{\ell}-x_{\ell}| \diff \bz \nonumber\\
& \quad \leq L \sum_{\ell = 1}^d \int_{[0,1]^d} \sum\limits_{\substack{k_1,\hdots,k_d \\ \sum_{j=1}^d k_j = k }} \frac{k!}{k_1! \hdots k_d !} \left( \frac{1}{d}\right)^k \prod_{m=1}^d      K_{k_m}^{uf}(0, |z_m-x_m|)|z_{\ell}-x_{\ell}| \diff \z \nonumber\\
& \quad \leq L \sum_{\ell = 1}^d \sum\limits_{\substack{k_1,\hdots,k_d \\ \sum_{j=1}^d k_j = k }} \frac{k!}{k_1! \hdots k_d !} \left( \frac{1}{d}\right)^k \prod_{m\neq \ell} \int_0^1 K_{k_m}^{uf}(0, |z_m-x_m|) \diff z_m \nonumber\\
& \qquad \qquad  \qquad \times   \int_0^1 K_{k_{\ell}}^{uf}(0,|z_{\ell}-x_{\ell}|)|z_{\ell}-x_{\ell}| \diff z_{\ell} \nonumber\\
&  \quad \leq L \sum_{\ell = 1}^d \sum\limits_{\substack{k_1,\hdots,k_d \\ \sum_{j=1}^d k_j = k }} \frac{k!}{k_1! \hdots k_d !} \left( \frac{2}{3}\right)^{k_l+1} \left( \frac{1}{d}\right)^k \prod_{m = 1}^d \int_0^1 K_{k_m}^{uf}(0,|z_m-x_m|) \diff z_m \nonumber\\
& \quad \quad \quad \textrm{(according to the second statement of Lemma \ref{lemme_uniform_rf_4}, see below)} \nonumber \\
 & \quad \leq \frac{L}{2^{k-d}} \sum_{\ell = 1}^d   \sum\limits_{\substack{k_1,\hdots,k_d \\ \sum_{j=1}^d k_j = k }} \frac{k!}{k_1! \hdots k_d !} \left(\frac{2}{3}\right)^{k_{\ell}+1} \left( \frac{1}{d}\right)^k,
 \label{equation_1_proof_bias_uniform}
\end{align}
according to the first statement of Lemma \ref{lemme_uniform_rf_4}. Still by Lemma \ref{lemme_uniform_rf_4} and using inequality (\ref{equation_1_proof_bias_uniform}), we have,
\begin{align*}
& \frac{\int_{[0,1]^d} K_k^{uf}(\0, |\z-\x|) |f(\z)-f(\bx)| \diff \bz }{\int_{[0,1]^d} K_k^{uf}(\0, |\z-\x|) \diff \bz} \\
& \quad  \leq \frac{L2^{2d+1}}{3} \sum_{\ell = 1}^d   \sum\limits_{\substack{k_1,\hdots,k_d \\ \sum_{j=1}^d k_j = k }} \frac{k!}{k_1! \hdots k_d !} \left(\frac{2}{3}\right)^{k_{\ell}} \left( \frac{1}{d}\right)^k.
\end{align*}
Taking the first term of the sum, we obtain
\begin{align*}
  \sum\limits_{\substack{k_1,\hdots,k_d \\ \sum_{j=1}^d k_j = k }} \frac{k!}{k_1! \hdots k_d !} \left(\frac{2}{3}\right)^{k_1} \left( \frac{1}{d}\right)^k 
 & \quad =  \sum_{k_1 = 0}^k  \left( \frac{2}{3d} \right)^{k_1} \left( 1 - \frac{1}{d}\right)^{k-k_1} \frac{k!}{k_1! (k-k_1)!}\\
 & \quad \leq    \left(1 - \frac{1}{3d}\right)^k.
\end{align*}
Finally, 
\begin{align*}
& \frac{\int_{[0,1]^d} K_k^{uf}(\0, |\z-\x|) |f(\z)-f(\bx)| \diff \bz }{\int_{[0,1]^d} K_k^{uf}(\0, |\z-\x|) \diff \bz} \leq \frac{L2^{2d+1}}{3}   \left(1 - \frac{1}{3d}\right)^k.
\end{align*}
\end{proof}

\begin{proof}[Proof of Theorem \ref{theoreme_consistency_uniform_forest_approximation}]

Let $\bx \in [0,1]^d$, $\|m\|_{\infty} = \sup\limits_{\bx \in [0,1]^d} |m(\bx)|$ and recall that
\begin{align*}
m_{\infty,n}^{uf}(\x) = & \frac{\sum_{i=1}^n Y_i K_k^{uf}(\0, |\X_i-\bx|)}{\sum_{i=1}^n  K_k^{uf}(\0, |\X_i-\bx|)}.
\end{align*}
Thus, letting 
\begin{align*}
& A_{n}(\x) = \frac{1}{n}\sum_{i=1}^n \left( \frac{Y_i K_k^{uf}(\0, |\X_i-\bx|)}{\E \big[ K_k^{uf}(\0, |\X-\bx|) \big]} - \frac{\E \big[ YK_k^{uf}(\0, |\X-\bx|) \big]}{\E \big[ K_k^{uf}(\0, |\X-\bx|) \big]} \right),\\
& B_{n}(\x) = \frac{1}{n} \sum_{i=1}^n \left( \frac{K_k^{uf}(\0, |\X_i-\bx|)}{\E \big[ K_k^{uf}(\0, |\X-\bx|) \big]} - 1 \right),\\
\textrm{and}~ & M_n(\x) = \frac{\E \big[ YK_k^{uf}(\0, |\X-\bx|) \big]}{\E \big[ K_k^{uf}(\0, |\X-\bx|) \big]}, 
\end{align*}
the estimate $m_{\infty,n}^{uf}(\x)$ can be rewritten as 
\begin{align*}
m_{\infty,n}^{uf}(\x) = \frac{M_n(\x) + A_n(\x)}{1 + B_n(\x)},
\end{align*}
which leads to 
\begin{align*}
m_{\infty, n}^{uf}(\x) - m(\x)& = \frac{  M_n(\x) - m(\x)+ A_n(\x) - B_n(\x) m(\x)}{1 + B_n(\x)}.
\end{align*}
Note that, according to Theorem \ref{bias_theorem}, we have
\begin{align*}
|M_n(\x) - m(\x)| & = \left|\frac{\E [ m(\bX) K_k^{uf}(\0, |\X-\bx|) ]}{\E [ K_k^{uf}(\0, |\X-\bx|) ]} + \frac{\E [ \varepsilon K_k^{uf}(\0, |\X-\bx|) ]}{\E [ K_k^{uf}(\0, |\X-\bx|) ]} - m(\bx) \right|\\
& \leq  \left|\frac{\E [ m(\bX) K_k^{uf}(\0, |\X-\bx|) ]}{\E [ K_k^{uf}(\0, |\X-\bx|) ]} - m(\bx) \right|\\
& \leq C_1 \left( 1 - \frac{1}{3d}\right)^k, 
\end{align*}
where $C_1 = L2^{2d+1}/3$. Take $\alpha\in ]0, 1/2]$. Let $\mathcal{C}_{\alpha}(\bx)$ be the event on which $\big\lbrace |A_n(\x)|, |B_n(\x)| \leq \alpha \big\rbrace$. On the event $\mathcal{C}_{\alpha}(\bx)$, we have
\begin{align*}
|m_{\infty,n}^{uf}(\x) - m(\x) |^2 & \leq 8 |M_n(\x) - m(\x) |^2 + 8 |A_n(\bx) - B_n(\bx) m(\x) |^2\\
& \leq 8C_1^2 \left( 1 - \frac{1}{3d} \right)^{2k} + 8\alpha^2 (1 + \|m\|_{\infty})^2.
\end{align*}
Thus, 
\begin{align}
\E [|m_{\infty, n}^{uf}(\x) - m(\x) |^2 \mathds{1}_{\mathcal{C}_{\alpha}(\bx)} ] \leq 8C_1^2 \left( 1 - \frac{1}{3d} \right)^{2k} + 8\alpha^2 (1 + \|m\|_{\infty})^2.\label{equation_proof_rate_consistency_uniform}
\end{align}
Consequently, to find an upper bound on the rate of consistency of $m_{\infty, n }^{uf}$, we just need to upper bound
\begin{align*}
\E \Big[ |\widetilde{m}_{\infty, n}^{uf}(\x) - m(\x)|^2 \mathds{1}_{\mathcal{C}^c_{\alpha}(\bx)} \Big] 
& \leq \E \Big[ |\max\limits_{1 \leq i \leq n} Y_i  + m(\x)|^2 \mathds{1}_{\mathcal{C}^c_{\alpha}(\bx)} \Big] \nonumber \\
& \quad \textrm{(since  $\widetilde{m}_{\infty, n}^{uf}$ is a local averaging estimate)}\nonumber\\
& \leq \E \Big[ |2 \|m\|_{\infty} + \max\limits_{1 \leq i \leq n} \varepsilon_i |^2 \mathds{1}_{\mathcal{C}^c_{\alpha}(\bx)} \Big] \nonumber\\
& \leq \left( \E \left[ 2\|m\|_{\infty} + \max\limits_{1 \leq i \leq n} \varepsilon_i\right]^4  \P \left[ \mathcal{C}^c_{\alpha}(\bx) \right] \right)^{1/2} \nonumber \\
& \quad \textrm{(by Cauchy-Schwarz inequality)} \nonumber \\
& \leq \left(  \left( 16 \|m\|_{\infty}^4 + 8 \E \Big[ \max\limits_{1 \leq i \leq n} \varepsilon_i\Big]^4 \right) \P \left[ \mathcal{C}^c_{\alpha}(\bx) \right] \right)^{1/2}. \nonumber
\end{align*}
Simple calculations on Gaussian tails show that one can find a constant $C>0$ such that for all $n$, 
\begin{align*}
\E \Big[ \max\limits_{1 \leq i \leq n} \varepsilon_i \Big]^4 \leq C (\log n)^2.
\end{align*}
Thus, there exists $C_2$ such that, for all $n >1$, 
\begin{align}
\E \Big[ |\widetilde{m}_{\infty, n}^{uf}(\x) - m(\x)|^2 \mathds{1}_{\mathcal{C}^c_{\alpha}(\bx)} \Big] 
& \leq C_2 (\log n) (\P \left[ \mathcal{C}^c_{\alpha}(\bx) \right] )^{1/2}. \label{proba_uniform_ineq}
\end{align}
The last probability $\P \left[ \mathcal{C}^c_{\alpha}(\bx) \right] $ can be upper bounded by using Chebyshev's inequality. Indeed, with respect to $A_n(\bx)$, 
\begin{align*}
\P \big[ |A_n(\bx)| > \alpha \big] 
& \leq \frac{1}{n \alpha^2 }\E \Bigg[ \frac{Y K_k^{uf}(\0, |\X-\bx|)}{\E [ K_k^{uf}(\0, |\X-\bx|) ]} - \frac{\E [ YK_k^{uf}(\0, |\X-\bx|) ]}{\E [ K_k^{uf}(\0, |\X-\bx|) ]} \Bigg]^2\\
& \leq \frac{1}{n \alpha^2 } \frac{1}{(\E [ K_k^{uf}(\0, |\X-\bx|) ])^2}\E \Big[ Y^2 K_k^{uf}(\0, |\X-\bx|)^2 \Big]\\
& \leq \frac{2}{n \alpha^2 } \frac{1}{(\E [ K_k^{uf}(\0, |\X-\bx|) ])^2}\bigg( \E [ m(\bX)^2 K_k^{uf}(\0, |\X-\bx|)^2 ] \\
& \qquad +\E [ \varepsilon^2 K_k^{uf}(\0, |\X-\bx|)^2 ]\bigg),\\
\end{align*}
which leads to
\begin{align*}
\P \big[ |A_n(\bx)| > \alpha \big] 
&\leq \frac{2(\|m\|_{\infty}^2+\sigma^2)}{n \alpha^2 } \frac{\E [ K_k^{uf}(\0, |\X-\bx|) ]}{(\E [ K_k^{uf}(\0, |\X-\bx|) ])^2}  \\
& \quad \textrm{(since $\sup\limits_{\bx, \bz \in [0,1]^d}K_k^{uf}(\0, |\z-\bx|)\leq 1$)}\\
& \leq \frac{M_1^2}{ \alpha^2 } \frac{2^k}{n}\\
& \quad \textrm{(according to the first statement of Lemma \ref{lemme_uniform_rf_4})},
\end{align*}
where $M_1^2=2^{d+1}(\|m\|_{\infty}^2+\sigma^2)$. Meanwhile with respect to $B_n(\bx)$, we have, still by Chebyshev's inequality, 
\begin{align*}
\P \big[ |B_n(\bx)| > \alpha \big] & \leq \frac{1}{n \alpha^2} \E \Bigg[ \frac{K_k^{uf}(\0, |\X_i-\bx|)}{\E [ K_k^{uf}(\0, |\X-\bx|)]}\Bigg]^2\\
& \leq \frac{1}{n \alpha^2} \frac{1}{\E [ K_k^{uf}(\0, |\X-\bx|) ]}\\
& \leq \frac{2^{k+d}}{n \alpha^2}.
\end{align*}
Thus, the probability of $\mathcal{C}_{\alpha}(\bx)$ is given by
\begin{align*}
\P \big[\mathcal{C}_{\alpha}(\bx) \big] &  \geq 1 - \P \big(|A_n(\x)| \geq \alpha \big) - \P \big(|B_n(\x)| \geq \alpha \big) \nonumber\\
& \geq 1 -  \frac{2^{k}}{n} \frac{M_1^2 }{\alpha^2}  - \frac{2^{k+d}}{n \alpha^2} \nonumber\\
& \geq 1 - \frac{2^{k}(M_1^2+2^d)}{n \alpha^2 }. 
\end{align*}
Consequently, according to inequality (\ref{proba_uniform_ineq}), we obtain 
\begin{align*}
\E \Big[ |\widetilde{m}_{\infty, n}^{uf}(\x) - m(\x)|^2 \mathds{1}_{\mathcal{C}^c_{\alpha}(\bx)} \Big] 
& \leq C_2 (\log n) \bigg(\frac{2^{k}(M_1^2+2^d)}{n \alpha^2 } \bigg)^{1/2}. 
\end{align*}
Then using inequality (\ref{equation_proof_rate_consistency_uniform}), 
\begin{align*}
& \E \Big[ \widetilde{m}_{\infty, n}^{uf}(\x) - m(\x) \Big]^2 \\
& \leq \E \Big[ |\widetilde{m}_{\infty, n}^{uf}(\x) - m(\x)|^2 \mathds{1}_{\mathcal{C}_{\alpha}(\bx)} \Big] + \E \Big[ |\widetilde{m}_{\infty, n}^{uf}(\x) - m(\x)|^2 \mathds{1}_{\mathcal{C}^c_{\alpha}(\bx)} \Big] \\
& \leq 8C_1^2 \left( 1 - \frac{1}{3d} \right)^{2k} + 8\alpha^2 (1 + \|m\|_{\infty})^2
+ C_2 (\log n) \bigg(\frac{2^{k}(M_1^2+2^d)}{n \alpha^2 } \bigg)^{1/2}.
\end{align*}
Optimizing the right hand side in $\alpha$, we get
\begin{align*}
\E \Big[ \widetilde{m}_{\infty, n}^{uf}(\x) - m(\x) \Big]^2 \leq 8C_1^2 \left( 1 - \frac{1}{3d} \right)^{2k} 
+ C_3 \bigg( \frac{(\log n)^2 2^k }{n} \bigg)^{1/3},
\end{align*}
for some constant $C_3>0$. The last expression is minimized for 
\begin{align*}
k = C_4 + \frac{1}{\log2 + \frac{2}{d}} \log \left( \frac{n}{(\log n)^2}\right),
\end{align*}
where $C_4 = - 3 \Big( \log2 + \frac{2}{d}\Big)^{-1} \log \left( \frac{d C_3 \log 2}{16 C_1^2}\right)$. Thus, there exists a constant $C_5 >0$ such that, for all $n>1$, 
\begin{align*}
& \E \left[ \widetilde{m}_{\infty, n}^{uf}(\x) - m(\x)\right]^2 \leq C n^{-2/(6 + 3d\log 2)} (\log n)^2.
\end{align*}

\end{proof}

\begin{lemme} \label{lemme_uniform_rf_4}
For all $k\in \mathds{N}$ and $x \in [0,1]$, 
\begin{itemize}
\item[$(i)$] \begin{align*}
\left(\frac{1}{2}\right)^{k_l+1} \leq \int_0^1 K_{k_l}^{uf}(0,|z_l-x_l|)\diff z \leq \left(\frac{1}{2}\right)^{k_l-1}.
\end{align*}

\item[$(ii)$] \begin{align*}
 \int_{[0,1]}   K_{k_l}^{uf}(0,|z_l-x_l|) |x_l - z_l| \textrm{d}z_l \leq \left(\frac{2}{3}\right)^{k_l+1} \int_{[0,1]}   K^{uf}_{k_l}(0,|z_l-x_l|) \textrm{d}z_l. 
\end{align*}
\end{itemize}

\end{lemme}

\begin{proof}[Proof of Lemma \ref{lemme_uniform_rf_4}]
Let $k_l\in \mathds{N}$ and $x_l\in [0,1]$. We start by proving $(i)$. According to Proposition \ref{lemme_foret_uniforme_expressiondunoyau}, the connection function of uniform random forests of level $k_l$ takes the form
\begin{align*}
\int_{[0,1]} K^{uf}_{k_l}(0,|z_l-x_l|) dz_l & =  \int_{-\log x_l }^{\infty} e^{-2u} \sum_{j=k_l}^{\infty} \frac{u^j}{j!} \diff u
							+ \int_{-\log (1-x_l) }^{\infty} e^{-2u} \sum_{j=k_l}^{\infty} \frac{u^j}{j!} \diff u\\
						& =  \sum_{j=k_l}^{\infty} \left( \frac{1}{2} \right)^{j+1} \int_{-2 \log x_l }^{\infty} e^{-u}  \frac{u^j}{j!} \diff u\\
							& \quad + \sum_{j=k_l}^{\infty} \left( \frac{1}{2} \right)^{j+1} \int_{-2\log (1-x_l) }^{\infty} e^{-u} \frac{u^j}{j!} \diff u\\
						& =  \sum_{j=k_l}^{\infty} \left( \frac{1}{2} \right)^{j+1} x_l^2 \sum_{i = 0}^j \frac{(- 2 \log x_l)^i}{i!} \\
							& \quad + \sum_{j=k_l}^{\infty} \left( \frac{1}{2} \right)^{j+1} (1-x_l)^2 \sum_{i = 0}^j \frac{(- 2 \log (1-x_l))^i}{i!}.
\end{align*}	
Therefore,
\begin{align*}
\int_{[0,1]} K_{k_l}^{uf}(0,|z_l-x_l|) \diff z_l \leq & \left(\frac{1}{2}\right)^{k_l-1},
\end{align*}
and
\begin{align*}
\int_{[0,1]} K_{k_l}^{uf}(0,|z_l-x_l|) \diff z_l \geq & \left( x_l^2 + (1-x_l)^2 \right)  \left(\frac{1}{2}\right)^{k_l} \geq \left(\frac{1}{2}\right)^{k_l+1}.
\end{align*}

Regarding the second statement of Lemma \ref{lemme_uniform_rf_4}, we have
\begin{align*}
&  \int_{[0,1]}   K_{k_l}^{uf}(0,|z_l-x_l|) |x_l - z_l| \textrm{d}z_l \\
& = \int_{[0,1]}    |x_l - z_l|^2 \sum_{j = k_l}^{\infty} \frac{(-\log |x_l - z_l|)^j}{j!}  \textrm{d}z_l \nonumber \\
& =  \int_{z_l \leq x_l} (x_l - z_l)^2 \sum_{j= k_l}^{\infty} \frac{\left( - \log |x_l - z_l| \right)^j}{j!} \textrm{d}z_l \\
& \quad \quad + \int_{z_l > x_l} (z_l-x_l)^2 \sum_{j= k_l}^{\infty} \frac{\left( - \log |x_l - z_l| \right)^j}{j!} \textrm{d}z_l \nonumber \\
& = \int_{[0,x_l]} v^2  \sum_{j= k_l}^{\infty} \frac{\left( - \log v \right)^j}{j!} dv + \int_{[0,1-x_l]} u^2  \sum_{j= k_l}^{\infty} \frac{\left( - \log u \right)^j}{j!} du\\
 & = \int_{-\log(x_l)}^{\infty} e^{-3w} \sum_{j= k_l}^{\infty} \frac{w^j}{j!} \textrm{d}w + \int_{- \log(1-x_l)}^{\infty} e^{-3 w}  \sum_{j= k_l}^{\infty} \frac{w^j}{j!} \textrm{d}w \nonumber\\
 &  =\frac{2}{3} \int_{-3\log(x_l)/2}^{\infty} e^{-2w} \sum_{j= k_l}^{\infty} \frac{(2v/3)^j}{j!} dv + \frac{2}{3}\int_{- 3\log(1-x_l)/2}^{\infty} e^{-2 v}  \sum_{j= k_l}^{\infty} \frac{(2v/3)^j}{j!} dv\\
& \leq \left(\frac{2}{3}\right)^{k_l+1} \left( \int_{-\log(x_l)}^{\infty} e^{-2w} \sum_{j= k_l}^{\infty} \frac{v^j}{j!} dv +  \int_{- \log(1-x_l)}^{\infty} e^{-2 v}  \sum_{j= k_l}^{\infty} \frac{v^j}{j!} \textrm{d}v \right) \nonumber \\
& \leq \left(\frac{2}{3}\right)^{k_l+1} \int_{[0,1]}   K^{uf}_{k_l}(0,|z_l-x_l|) \textrm{d}z_l. 
\end{align*}

\end{proof}

\section{Acknowledgments}

We would like to thank Arthur Pajot for his great help in the implementation of KeRF estimates.

\bibliography{biblio-sbv}

\end{document}